\documentclass[12pt,reqno]{amsart}
\usepackage{hyperref}

\usepackage{amsfonts, amssymb,amsmath,amsthm,enumerate}
\usepackage{color}
\usepackage[margin=1.5in]{geometry}

\RequirePackage{mathrsfs} \let\mathcal\mathscr

\newtheorem{theorem}{Theorem}[section]
\newtheorem{lemma}[theorem]{Lemma}
\newtheorem{prop}[theorem]{Proposition}
\newtheorem{corollary}[theorem]{Corollary}

\theoremstyle{definition}
\newtheorem*{ack}{Acknowledgements}
\newtheorem{rem}[theorem]{Remark}
\newtheorem{example}[theorem]{Example}

\newtheorem*{hyp}{Hypothesis-$\rho$}

\newcommand{\sfl}{\mathsf{\Lambda}}

\renewcommand{\d}{\mathrm{d}}
\renewcommand{\phi}{\varphi}

\newcommand{\card}{\#}

\newcommand{\PP}{\mathbb{P}}
\renewcommand{\AA}{\mathbb{A}}
\newcommand{\FF}{\mathbb{F}}
\newcommand{\ZZ}{\mathbb{Z}}

\newcommand{\NN}{\mathbb{N}}
\newcommand{\QQ}{\mathbb{Q}}
\newcommand{\RR}{\mathbb{R}}
\newcommand{\CC}{\mathbb{C}}

\renewcommand{\rho}{\varrho}

\renewcommand{\leq}{\leqslant}
\renewcommand{\le}{\leqslant}
\renewcommand{\geq}{\geqslant}
\renewcommand{\ge}{\geqslant}
\renewcommand{\bar}{\overline}

\newcommand{\ma}{\mathbf}

\newcommand{\m}{\mathbf{m}}
\newcommand{\M}{\mathbf{M}}
\newcommand{\B}{\mathbf{B}}
\newcommand{\A}{\mathbf{A}}

\newcommand{\x}{\mathbf{x}}
\newcommand{\y}{\mathbf{y}}
\renewcommand{\c}{\mathbf{c}}
\newcommand{\f}{\mathbf{f}}
\renewcommand{\v}{\mathbf{v}}
\renewcommand{\u}{\mathbf{u}}

\renewcommand{\a}{\mathbf{a}}

\newcommand{\bt}{\mathbf{t}}

\newcommand{\ve}{\varepsilon}

\newcommand{\bla}{\boldsymbol{\lambda}}

\newcommand{\bxi}{\boldsymbol{\xi}}

\DeclareMathOperator{\rank}{rank}

\DeclareMathOperator{\Mod}{mod} 
\renewcommand{\bmod}[1]{\,(\Mod{#1})}

\newcommand{\Hyp}{Hypothesis-$\rho$}

\numberwithin{equation}{section}

\begin{document}
\title[Power-free values of
  polynomials on symmetric varieties]{Power-free values of 
  polynomials on \\ symmetric varieties}

\date{\today}

\author{T.D. Browning}
\author{A. Gorodnik}

\address{School of Mathematics\\
University of Bristol\\ Bristol\\ BS8 1TW\\ UK}
\email{t.d.browning@bristol.ac.uk, a.gorodnik@bristol.ac.uk}

\begin{abstract} 
Given a symmetric variety $Y$ defined over $\QQ$ and 
a non-zero polynomial with integer coefficients, 
we use  techniques from homogeneous dynamics to establish conditions under which the polynomial 
can be made $r$-free for a Zariski dense set of integral points on 
$Y$. We also establish an asymptotic counting formula for this set.
In the special case that $Y$ is a 
 quadric hypersurface, we give explicit bounds on the size of $r$ by combining the argument with 
 a uniform upper bound for the density of integral points on general affine quadrics defined over $\QQ$.
\end{abstract}

\subjclass[2010]{11N32 (11D09, 11D45, 20G30)}

\maketitle
\setcounter{tocdepth}{1}
\tableofcontents

\thispagestyle{empty}

\section{Introduction}\label{s:intro}

Given a polynomial with integer coefficients, the problem of determining whether or not 
it takes infinitely many square-free values has long been a central concern in analytic number theory. 
More generally, one can ask for  $r$-free values, for any $r\geq 2$, 
where an integer is said to be {\em $r$-free} if it is not divisible by $p^r$ for any prime $p$. 
In this paper we initiate an investigation of $r$-free values of  polynomials whose arguments run over {\em thin sets}.

Let $Y\subset \AA^n$ be an affine variety defined by a system of polynomial equations with integer coefficients,
with $Y(\ZZ)\neq \emptyset$,  and let 
$f\in \ZZ[X_1,\ldots,X_n]$  be a polynomial.  Nevo and Sarnak \cite{n-s}
 define the {\em saturation number} $r(Y,f)$ to be the least positive integer $r$ such that the set of $\x\in Y(\ZZ)$, for which $f(\x)$ has at most $r$ prime factors, is Zariski dense in $Y$.
They show that $r(Y,f)$ is  finite whenever $Y$ is a principal homogeneous space of a semisimple algebraic group and 
 $f$ is ``weakly primitive''. 
 In a similar spirit, we can define the {\em permeation number}
 $r^\square (Y,f)$ 
to be the least  integer $r\geq 2$ such that the set
\begin{equation}
\label{eq:rf}
\{\x\in Y(\ZZ): \hbox{ $f(\x)$ is $r$-free}\}
\end{equation}
is Zariski dense in $Y$. The following natural condition becomes relevant in this setting.
We say that the polynomial $f$ has an  {\em 
$r$-power divisor on $Y$} if  there is a prime $p$ such that $p^r\mid f(\x)$ for every $\x\in Y(\ZZ_p)$, where $\ZZ_p$ denotes the ring of $p$-adic integers. 
It is clear that when the polynomial $f$ has an $r$-power divisor on $Y$,
the set \eqref{eq:rf} is empty. On the other hand, in this paper we show that 
for some classes of varieties and sufficiently large $r$, the set \eqref{eq:rf}
is Zariski dense provided that $f$ has no $r$-power divisor on $Y$.
We also establish an asymptotic counting formula describing the distribution of this set.

One of the earliest examples arises in work of  Erd\H{o}s \cite{sfree}, who showed that 
$
r^\square(\AA^1,f)\leq d-1,
$ 
provided that $f$ has degree $d$ and contains no 2-power divisors.
Assuming the truth of the $abc$-conjecture, Poonen \cite{poonen} has established the equality
$r^\square(\AA^n,f)=2$ for any polynomial $f\in \ZZ[X_1,\dots,X_n]$ without 2-power divisors.
Our main result establishes finiteness of the permeation number 
$r^\square(Y,f)$
for generic $f$ and 
a general class of {\em symmetric varieties}  $Y\subset \mathbb{A}^n$  over $\QQ$.
Thus, let $G$ be a connected semisimple algebraic group defined over $\QQ$
and let $\iota: G\to \hbox{GL}_n$ be an almost faithful linear representation, also  defined over $\QQ$.
We assume that $G$ acts transitively on $Y$ and that $Y(\ZZ)$ is non-empty.
Then $Y\simeq G/L$, where $L$ is an algebraic subgroup of $G$ defined over $\QQ$.
The symmetric varieties dealt with here are assumed to satisfy 
the following properties:
\begin{itemize}
	\item[(i)] the group $L$ is a symmetric subgroup of $G$ (i.e.\  the Lie algebra of $L$ is equal to 
	the set of fixed points of a non-trivial involution defined over $\QQ$);
	\item[(ii)] the connected component of $L$ has no non-trivial $\QQ$-rational characters;
	\item[(iii)] the group $G$ is $\QQ$-simple and  simply connected; and
	\item[(iv)] the group $G(\RR)$ has no compact factors.
\end{itemize}
It is known that the set $Y(\ZZ)$ of integral points can be parametrised 
by orbits of the arithmetic group 
$
\Gamma=\iota^{-1}(\hbox{GL}_n(\ZZ)).
$
According to  Borel and Harish-Chandra \cite{bh},
the set $Y(\ZZ)$ is a union of finitely many $\Gamma$-orbits.  
This allows us to study the set of $r$-free points using techniques from homogeneous dynamics.

It is very natural to demand that  $f\in \ZZ[X_1,\dots,X_n]$ be devoid of $r$-power divisors on $Y$.  It turns out that our argument also  requires  knowledge of  the arithmetic function
\begin{equation}\label{eq:def-rho}
\rho(\ell)=
\#\left\{
\x\in Y(\ZZ/\ell\ZZ): f(\x)\equiv 0 \bmod{\ell}\right\},
\end{equation}
for  $\ell\in \NN$.  This function is multiplicative, by the Chinese remainder theorem, and we can only handle  $f$ for which the  prime power constituents 
of $\rho(\ell)$ 
satisfy  the following assumption.

\begin{hyp}
For any $r\geq 1$, there exists a constant $C_{Y,f,r}>0$, depending on $r$ and the coefficients of $Y$ and $f$,  
 such that $\rho(p^r)\leq C_{Y,f,r}p^{r(\dim(Y)-1)}$, for any prime $p$.
\end{hyp}

Let $Z$ denote the variety $Y\cap\{f=0\}$.
When $r=1$ the upper bound for $\rho(p)$ in \Hyp~follows from the   Lang--Weil estimate if $Z$ has codimension $1$ in $Y$.
If we further assume that 
$Z$ is a non-singular affine variety of codimension~$1$ in $Y$, then  \Hyp~follows from an application of  Hensel's lemma.
Since $Y$ is non-singular, it is worth emphasising that \Hyp~holds for 
generic choices of $f$.
We shall see that \Hyp~is also satisfied for quadric hypersurfaces (see Lemma \ref{lem:rho'} below).

Bearing this in mind, we may now record our first main result. 

\begin{theorem}\label{t:symmetric}
Let $Y\simeq G/L\subset \AA^n$ be a symmetric variety over $\QQ$ satisfying  (i)--(iv), 
with $Y(\ZZ)\neq \emptyset$.
Assume that $f\in \ZZ[X_1,\dots,X_n]$ satisfies Hypothesis-$\rho$.
Then $r^\square(Y,f)<\infty$.
\end{theorem}

More precisely, we show that there exists $r_0$ such that for $r\ge r_0$, if the set \eqref{eq:rf}
is not empty, then it is Zariski dense in $Y$.
Moreover, if $L$ is additionally assumed to be semisimple and simply connected, then for $r\ge r_0$,
the set \eqref{eq:rf} is Zariski dense provided only that $f$ does not have $r$-power divisors on $Y$.
The value of $r_0$ is not made explicit in this work, but
it can be estimated using our method.
It depends on $\dim(G),\deg(f)$ and on 
the uniform spectral gap property that was  shown by Burger--Sarnak \cite{bs} and Clozel \cite{c} to be 
enjoyed by the action of 
each non-compact simple factor of $G(\RR)$ on the congruence quotients $G(\RR)/\Gamma_\ell$, 
where
\begin{equation}
\label{eq:cong}
\Gamma_\ell=\{\gamma\in \Gamma:\, \iota(\gamma)=id\;\hbox{mod}\; \ell\}.
\end{equation}
Although we shall not pursue it here, we note that 
the arguments in this paper could also be used to generalise the finite saturation results of Nevo and Sarnak \cite{n-s} to 
 a broader class of symmetric varieties. 

Our argument also allows us to establish an asymptotic formula for the number of $r$-free points.
For $r\geq 2$ and a polynomial  $f\in \ZZ[X_1,\dots,X_n]$, define \begin{equation}\label{eq:london0}
N_r(Y,f;H)=\card\{\x\in Y(\ZZ): |\x|\leq H, ~\text{$f(\x)$ is $r$-free}\},
\end{equation}
where $|\x|=\max_{1\leq i\leq n}|x_i|$. 
The main term in the asymptotic formula for $N_r(Y,f;H)$ 
will involve a product of local densities which we proceed to define here.
To define the real density, we assume that the variety $Y$ is the zero locus of a family of polynomials 
$f_1,\ldots,f_\ell \in \ZZ[X_1,\ldots, X_n]$ that satisfy
\begin{equation}
\label{eq:smooth}
\hbox{rank}\left(\frac{\partial f_i}{\partial X_j}\right)=n-\dim(Y)
\end{equation}
everywhere on $Y$. Then we define the real density by
\begin{equation}\label{eq:real}
\mu_\infty(Y;H)=\lim_{\epsilon\rightarrow 0} \frac{1}{\epsilon^\ell}\int_{\substack{
		|\x|\le H\\  |f_1(\x)|,\ldots, |f_\ell(\x)|<\epsilon/2}} \d \x.
\end{equation}
For each prime $p$, the  $p$-adic density is
\begin{equation}\label{eq:padic}
\hat \mu_p(Y,f,r)=\lim_{t\rightarrow \infty} p^{-t\dim(Y)}\#\{\x\in Y(\ZZ/p^t\ZZ):  ~p^r\nmid f(\x)\}.
\end{equation}
We also define the  Euler product 
\begin{equation}\label{eq:euler0}
\mathfrak{S}(Y,f,r)=\prod_{p< \infty}\hat \mu_p(Y,f,r).
\end{equation}
If $L$ is semisimple, then under Hypotheses-$\rho$, this product converges absolutely.

With this notation, we prove the following result. 

\begin{theorem}\label{t:symmetric2}
	Let $Y\simeq G/L\subset \AA^n$ be a symmetric variety over $\QQ$ satisfying  (i)--(iv).
	We assume that $L$ is semisimple and simply connected. 
	Let $f\in \ZZ[X_1,\dots,X_n]$ be a polynomial satisfying Hypothesis-$\rho$.
	Then for all sufficiently large $r$, there exists a  constant    $\delta>0$
	such that 
	$$
	N_r(Y,f;H)=\mathfrak{S}(Y,f,r) \mu_{\infty}(Y;H)
	+O_{r}(\mu_{\infty}(Y;H)^{1-\delta}).
	$$
	Moreover, $\mathfrak{S}(Y,f,r)>0$ provided that $f$ does not have $r$-power divisors on $Y$.
\end{theorem}

Throughout our work, unless stated otherwise, we will allow our implied constants to depend on the polynomial $f$ and the variety $Y$, which are  considered to be fixed once and for all. Any further  dependence will be  explicitly indicated by  appropriate subscripts. In Theorem \ref{t:symmetric2}, for example, 
the implied constant in the error term is allowed to depend on $r$,  on  $f$ and  
on the polynomials defining $Y$.

We also establish an asymptotic formula for $N_r(Y,f;H)$ when $L$ is not assumed to be a semisimple simply connected group (see Theorem \ref{th:general} and Remark \ref{r:BM} below). However, without this assumption, the variety $Y\simeq G/L$ may fail to satisfy the local-to-global principle. Moreover,  the definition of the Euler product \eqref{eq:euler0} requires the introduction of additional convergence factors, so that the main term in the asymptotic formula becomes significantly more involved.

Our remaining results are concerned with producing explicit upper bounds for $r^\square(Y,f)$ for  quadric hypersurfaces.
For   $n\geq 3$, let
 $Q\in \ZZ[X_1,\ldots,X_n]$
be a non-singular indefinite quadratic form and let 
$m$ be a non-zero integer. 
We shall always assume that 
 $-m\det (Q)$ is not the square of an integer when $n=3$.
We  let $Y\subset \AA^n$ denote the affine quadric  
\begin{equation}\label{eq:green}
Q(X_1,\dots,X_n)=m.
\end{equation}
We observe that our general results (Theorem \ref{t:symmetric} and \ref{t:symmetric2}) are applicable in this setting:

\begin{rem}\label{r:quad}
	The assumptions (i)--(iv) are satisfied in the setting of quadric hypersurfaces \eqref{eq:green}
	(with a possible exception of $G$ being $\QQ$-simple when $n=4$, which we discuss separately in Remark \ref{r:22}).
	In the case of quadric hypersurfaces, 
	$G=\hbox{Spin}(Q)$ is the spinor group of $Q$  and  $\iota:G\to \hbox{GL}_n$ is the standard representation 
	of the spinor group of $Q$. We let  $\Gamma=\iota^{-1}(\hbox{GL}_n(\ZZ))$ and 
	$L=\hbox{Stab}_G(\x_0)$, with $\x_0\in Y(\QQ)$.
	Thus 
	$$
	\dim(G)=\tfrac{1}{2}n(n-1)\quad \text{ and }\quad  \dim(L)=\tfrac{1}{2}(n-1)(n-2).
	$$ 
	Moreover,  $L$ is a symmetric subgroup of $G$ and
	$L\simeq \hbox{Spin}(Q|_V)$, where $V$ is the orthogonal complement of $\x_0$.
	In particular, when $n\ge 4$, it follows that  $L$ is a semisimple simply connected algebraic group, and when $n=3$, $L$ is a one-dimensional torus.
	We observe that $\det (Q|_V)=\det(Q)/m$, so that  when $n=3$, $Q|_V$ is equivalent to the quadratic form
	$x^2+m\det(Q)y^2$. Hence, if $-m\det(Q)$ is not a square, $L$ is anisotropic over $\QQ$,
	and the assumption (ii) is satisfied.
	The group $G=\hbox{Spin}(Q)$ is simply connected, so that $G(\RR)$ is connected.
	Moreover, $G(\RR)\simeq \hbox{Spin}(r_1,r_2)$, where $(r_1,r_2)$ is the signature
	of the quadratic form $Q$. Since $Q$ is assumed to be isotropic over $\QQ$,  
	$G(\RR)$ is not compact. It is simple unless $(r_1,r_2)=(2,2)$,
	in which case $G(\RR)\simeq \hbox{SL}_2(\RR)\times \hbox{SL}_2(\RR)$.
	Hence, $G(\RR)$ has no compact factors.
	It also follows that $G$ is $\QQ$-simple, unless $(r_1,r_2)=(2,2)$.
	We discuss the case  $(r_1,r_2)=(2,2)$ in Remark \ref{r:22}.
\end{rem}

Thus  $r^\square(Y,f)<\infty$ for quadratic hypersurfaces \eqref{eq:green} with $Y(\ZZ)\neq 0$ and any integral polynomial $f$ satisfying \Hyp.
When $n\geq 4$ and certain  necessary conditions are met, Baker~\cite{baker1} has used a 
variant of the Hardy--Littlewood circle method  to show that there
exist infinitely many points $\x\in Y(\ZZ)$ with all the coordinates $x_i$ square-free, provided that the obvious local conditions are satisfied. 
A modification of Baker's argument would easily give $r^\square(Y,X_i)=2$, for any $i\in \{1,\dots,n\}$, provided that $n\geq 4$. 
In this paper, we give explicit bounds on $r^\square(Y,f)$ 
and establish an asymptotic formula for $N_r(Y,f;H)$
when $f$ is an arbitrary  non-singular form.

We define the Euler product as in \eqref{eq:euler0}.
When $n\ge 4$, this product is absolutely convergent, and positive provided that 
$f$ does not have $r$-power divisors on $Y$.
It is only conditionally convergent when $n=3$ and $-m\det(Q)$ is not a square.

Our first  result specific to  quadrics concerns the asymptotic behaviour of 
$N_r(Y,f;H) $ in the easier case  $n\ge 4$.

\begin{theorem}\label{thm:2large}
	Let $n\geq 4$ and let $Y\subset \AA^n$ be the quadric hypersurface \eqref{eq:green}.
	Assume that $f$ is a non-singular form of degree $d\geq 2$ and let
	$
	r\geq dn^2/(n-1).
	$
	Then there exists a  constant    $\delta>0$
	such that 
	$$
	N_r(Y,f;H)=\mathfrak{S}(Y,f,r) \mu_{\infty}(Y;H)
	 +O_{r}(H^{n-2-\delta}).
	$$
	Moreover, $\mathfrak{S}(Y,f,r)>0$ provided that $f$ does not have $r$-power divisors on $Y$.
\end{theorem}

Here, we note that $\mu_{\infty}(Y;H)\sim \mu_{\infty}(Y)H^{n-2}$, as $H\to \infty$, for some constant $\mu_{\infty}(Y)>0$.

The case $n=3$ is much harder because quadric surfaces may fail to satisfy the
local-to-global principle. This phenomenon can be analysed using a cohomological invariant
introduced by Borovoi and Rudnick \cite{borovoi,borovoi'}.
This invariant is a locally constant function 
$$
\delta:Y(\A)\to\{0,2\},
$$ 
defined on the adelic space
$
Y(\A)=Y(\RR)\times {\prod}^\prime_{p< \infty} Y(\QQ_p).
$
If $Y(\QQ)=\emptyset$, then $\delta\equiv 0$.
Otherwise, we fix $\x_0\in Y(\QQ)$. Let $G=\hbox{Spin}(Q)$ be the spinor group of $Q$. Then $G$ acts transitively on $Y$ and $G(\A)$ acts on $Y(\A)$, but the latter action is not transitive.
Orbits $\mathcal{O}_{\A}$ for this action are open in $Y(\A)$ and they are restricted direct products
\begin{equation}
\label{eq:aad}
\mathcal{O}_{\A}={\prod}^\prime_{p\le \infty} \mathcal{O}_p,
\end{equation}
where each $\mathcal{O}_p$ is an open orbits of $G(\QQ_p)$ in $Y(\QQ_p)$. We define
$$
\nu_p(\mathcal{O}_p)=\begin{cases}
+1 & \text{if $\mathcal{O}_p=G(\QQ_p)\x_0$},\\
-1 & \text{if $\mathcal{O}_p\ne G(\QQ_p)\x_0$}.
\end{cases}
$$
We note that $\nu_p(\mathcal{O}_p)=1$ for almost all $p$. Let
$$
\nu(\mathcal{O}_\A)=\prod_{p\le\infty} \nu_p(\mathcal{O}_p).
$$
One can show that $\nu$ is independent of the choice of $\x_0\in Y(\QQ)$.
The function is extended  to elements of $Y(\A)$ by setting $\nu(\x)=\nu(G(\A)\x)$, for any $\x\in Y(\A).$ Next,  we set 
$
\delta=1+\nu.
$
This defines a locally constant function on $Y(\A)$.
It was shown in \cite{borovoi,borovoi'} that $\delta(\mathcal{O}_\A)=0$ if and only if $\mathcal{O}_\A$ contains no rational points. (This  theory can be also interpreted in terms of the {\em integral Brauer--Manin obstruction}, as worked out by Colliot-Th\'el\`ene and Xu \cite{cx}.)

As in \eqref{eq:real}--\eqref{eq:padic}, we define local densities of adelic orbits \eqref{eq:aad}. Since orbits of $G(\RR)$ in $Y(\RR)$ are open and connected,
they are equal to connected components of the quadratic surface $Y(\RR)$.
The real density  is defined by
\begin{equation}\label{eq:real2}
\mu_\infty(\mathcal{O}_{\A};H)=\lim_{\epsilon\rightarrow 0} \frac{1}{\epsilon}\int_{\substack{
		\x\in\mathcal{O}^*_\infty,\,|\x|\le H\\  |Q(\x)-m|<\epsilon/2}} \d \x, 
\end{equation}
where $\mathcal{O}^*_\infty$ is a fixed neighbourhood of $\mathcal{O}_\infty$ which does not intersect the other connected components of $Y(\RR)$.
The $p$-adic densities are defined by
$$
\hat \mu_p(\mathcal{O}_{\A},f,r)=\lim_{t\rightarrow \infty} p^{-t(n-1)}\#\{\x\in \mathcal{O}_p\cap Y(\ZZ_p)\;\hbox{mod}\; p^t:  ~p^r\nmid f(\x)\}.
$$
We note that $ Y(\ZZ_p)\subset \mathcal{O}_p$ for almost all $p$. Hence, $\hat \mu_p(\mathcal{O}_{\A},f,r)=\hat \mu_p(Y,f,r)$ for almost all $p$.
We also define the  Euler product 
$$
\mathfrak{S}(\mathcal{O}_{\A},f,r)=\prod_{p<\infty}\hat \mu_p(\mathcal{O}_{\A},f,r),
$$
which differs from the Euler product \eqref{eq:euler0} only at finitely many factors. 

For $n=3$ we have the following result.

\begin{theorem}\label{thm:2}
	Let $n=3$ and let $Y\subset \AA^3$ be the quadric surface \eqref{eq:green}.
	Assume that $f$ is a non-singular form of degree $d\geq 2$ and let    $r\geq 4d^2+\frac{4}{3}d$.
	Then there exists $\delta>0$  such that 
	$$
	N_r(Y,f;H)=\sum_{\mathcal{O}_{\A}\subset Y(\A)}\delta(\mathcal{O}_{\A})\mathfrak{S}(\mathcal{O}_{\A},f,r) \mu_{\infty}(\mathcal{O}_{\A};H)
	 +O_{r}(H^{1-\delta}),
	$$
where the sum is taken over finitely many orbits $\mathcal{O}_{\A}$
	that have non-trivial intersection with $Y(\RR)\times \prod_{p<\infty} Y(\ZZ_p)$.
\end{theorem}

The following is an immediate consequence of Theorems \ref{thm:2large} and \ref{thm:2},  giving 
an explicit version of Theorem \ref{t:symmetric} in the setting of quadric hypersurfaces.

\begin{corollary}
	Let $n\geq 3$ and let $Y\subset \AA^n$ be the quadric hypersurface \eqref{eq:green}.
Assume that $f$ is a non-singular form of degree $d\geq 2$.
We set 
$$
r_0(n,d)= \begin{cases}
 	4d^2+\frac{4}{3}d & \text{if $n=3$,}\\
 	dn^2/(n-1) & \text{if $n\geq 4$.}
 \end{cases}
$$
When $n\ge 4$, we denote by $r_0(Y,f)$ the least $r$ such that 
$f$ has no $r$-power divisors on $Y$.
When $n=3$, we denote by $r_0(Y,f)$ the least $r$ such that
there exists $\x\in Y(\ZZ)$ with $f(\x)$ being r-free.
Then 
 $$r^\square(Y,f)\leq \max\{ r_0(n,d), r_0(Y,f)\}.
 $$
 \end{corollary}

	We illustrate Theorem \ref{thm:2} with some examples borrowed from the work of Borovoi and Rudnick \cite{borovoi,borovoi'}.

\begin{example}[\S 6.4.1 in \cite{borovoi}]
		Let 
		$
		Q(X_1,X_2,X_3)=-9X_1^2+2X_1X_2+7X_2^2+2X_3^2
		$ and  let $m=1$.
		Then the equation defining $Y$ can be rewritten
		$$
		(X_2-X_1)(9X_1+7X_2)=1-2X_3^2.
		$$
		One easily checks that $(-\frac{1}{2},\frac{1}{2},1)$ and $(\frac{1}{3},0,1)$ are points in $Y(\QQ)$, so that 
		there are solutions over $\ZZ_p$ for every prime $p$ with $X_3=1$. 
		Moreover, in the asymptotic formula for $N_2(Y,X_3;H)$ one finds that $\mathfrak{S}(Y,X_3,2)>0$. On the other hand, $Y(\ZZ)=\emptyset$.
		From the point of view of Theorem \ref{thm:2} this means that
		$\delta(\mathcal{O}_{\A})=0$
		for all adelic orbits $\mathcal{O}_{\A}$
		that have non-trivial intersection with $Y(\RR)\times \prod_{p<\infty} Y(\ZZ_p)$.
\end{example}

\begin{example}[\S 4 in \cite{borovoi'}]
Let us assume that the hyperboloid $Y(\RR)$ has two connected components.
		Consider the involution 
		$$
		\iota:Y(\A)\to Y(\A), \quad (\y_\infty,\y_f)\mapsto (-\y_\infty,\y_f).
		$$
		It is clear that $\iota$ 
		maps orbits $\mathcal{O}_\A$ to orbits, and it follows from 
		the definition of the invariant $\nu$ that $\nu(\iota(\mathcal{O}_\A))=-\nu(\mathcal{O}_\A)$. Hence,
		$
		\delta(\iota(\mathcal{O}_\A))+\delta(\mathcal{O}_\A)=2.
		$
		Moreover, 
		$$
		\mathfrak{S}(\iota(\mathcal{O}_{\A}),f,r)=
		\mathfrak{S}(\mathcal{O}_{\A},f,r) \quad \text{ and } \quad 
		\mu_{\infty}(\iota(\mathcal{O}_{\A});H)=\mu_\infty(\mathcal{O}_{\A};H).
		$$
		Hence, Theorem \ref{thm:2} implies that
			\begin{align*}
			N_r(Y,f;H)&=\sum_{\mathcal{O}_{\A}\subset Y(\A)}\mathfrak{S}(\mathcal{O}_{\A},f,r) \mu_{\infty}(\mathcal{O}_{\A};H)
			+O_{r}(H^{1-\delta})\\
			&=\mathfrak{S}(Y,f,r) \mu_{\infty}(Y;H)
			+O_{r}(H^{1-\delta}).
			\end{align*}
		In this case the main term happens to satisfy the Hardy--Littlewood prediction
		even though the integral points are far from being equidistributed with respect to
		the orbits $\mathcal{O}_\A$. Indeed, among $\iota(\mathcal{O}_\A)$ and $\mathcal{O}_\A$,
		only one of the sets contains integral points.
\end{example}

\begin{example}[\S 3 in \cite{borovoi'}]
Assume that $Y(\ZZ)\ne \emptyset$, but there exists a quadratic form 
		in the genus of $Q$ which does not represent $m$ over $\ZZ$.  In this case,
		Theorem \ref{thm:2} gives
			\begin{align*}
			N_r(Y,f;H)
			&=2\mathfrak{S}(Y,f,r) \mu_{\infty}(Y;H)
			+O_{r}(H^{1-\delta}).
			\end{align*}
		Indeed, in this case it was was proved in \cite{borovoi'} that $Y(\RR)\times \prod_{p<\infty} Y(\ZZ_p)$ is contained in a single orbit $\mathcal{O}_\A$
		with $\delta(\mathcal{O}_\A)=2$.
	\end{example}

We can do better than Theorems \ref{thm:2large} and \ref{thm:2} when $f$ is linear,
in which case one can  actually produce  an asymptotic formula for $N_r(Y,f;H)$, for all $r\geq 2$.
This has the following outcome.

\begin{theorem}\label{thm:1}
Let $n\geq 3$ and let $Y\subset \AA^n$ be the quadric hypersurface \eqref{eq:green}, 
with $Y(\ZZ)\neq \emptyset$.
Assume that $f$ is a linear form having no 2-power divisors on $Y$. 
When $n=3$, we additionally assume that there exists $\x\in Y(\ZZ)$ such that $f(\x)$ is square-free. Then 
 $r^\square(Y,f)=2$.
\end{theorem}

We now return  to the setting of a general
 symmetric variety
 $Y\simeq G/L\subset \AA^n$ defined over $\QQ$, with $G,L$ satisfying (i)--(iv).
 Let $f\in \ZZ[X_1,\dots,X_n]$ be a polynomial that satisfies \Hyp. Recall  the counting function 
$N_r(Y,f;H)$ from 
\eqref{eq:london0}.   The igniting spark in its analysis is provided by the indicator function
$$
\sum_{k^r\mid N}
\mu(k)=
\begin{cases}
1 & \text{if $N$ is $r$-free,}\\
0 & \text{otherwise,}
\end{cases}
$$
where $N\in \ZZ$ is non-zero and $\mu$ is the M\"obius function. 
Thus 
\begin{equation}\label{eq:spud}
N_r(Y,f;H)=\sum_{k=1}^\infty \mu(k) 
\card \left\{\x\in Y(\ZZ): 
\begin{array}{l}
|\x|\leq H\\
0\neq f(\x)\equiv 0 \bmod{k^r}
\end{array}{}
\right\}.
\end{equation}
Since $f$ has degree $d$ it is clear that 
the summand vanishes 
unless $k\ll H^{d/r}$.  
Moreover, since $Y(\ZZ)$ consists of finitely many $\Gamma$-orbits, 
we may break the sum into  residue classes modulo $k^r$ and find 
that  estimating it reduces to 
estimating
\begin{equation}
\label{eq:orb_num}
\#\{\x\in \Gamma \y: |\x|\leq H, ~\x\equiv \bxi \bmod{k^r}\},
\end{equation}
for given  $\y\in Y(\ZZ)$
and
given $\bxi \in Y(\ZZ/k^r\ZZ)$ such that $f(\bxi)\equiv 0\bmod{k^r}$.
The sets
 $\{\x\in \Gamma \y: \x\equiv \bxi \bmod{k^r}\}$
are finite unions of $\Gamma_{k^r}$-orbits, where $\Gamma_\ell$ is given by \eqref{eq:cong} for $\ell\in \NN$.
Thus 
the investigation of \eqref{eq:orb_num} reduces 
to establishing an asymptotic formula for 
$\#\{ \x \in \Gamma_\ell \y:\, |\x|\le H\}$,
as $H\to\infty$, which is uniform in $\ell$. This estimate is the focus of \S \ref{s:counting} and lies at the heart of this paper (see Theorem \ref{th:counting}).  The error term  involves a polynomial dependence on $\ell$, meaning that it is only useful  for handling  the contribution 
to  $N_r(Y,f;H)$ from sufficiently small values of $k^r$.  

By taking $r$ sufficiently large we can ensure that $k$ is an arbitrarily small power of $H$. In this way,
on observing that 
$$
\card \left\{\x\in Y(\ZZ): |\x|\leq H, ~ k^r\mid f(\x)
\right\}\leq 
\card \left\{\x\in Y(\ZZ): |\x|\leq H, ~ k^2\mid f(\x)
\right\},
$$
it is possible to reapply the results from \S \ref{s:counting} with $\ell=k^2$, in order to show that the larger values of $k^r$ make a negligible contribution to $N_r(Y,f;H)$. This 
summarises our strategy behind the proof of Theorem \ref{t:symmetric2}.
The proof of Theorem \ref{t:symmetric} requires a generalisation of Theorem \ref{t:symmetric2},
which gives an asymptotic formula for 
the number $r$-free points lying on a given adelic orbit (see \S \ref{s:gene}).

Our proof of Theorems \ref{thm:2large}, \ref{thm:2} and \ref{thm:1} 
gets under way in \S \ref{s:over} and 
relies on a more efficient method for handling the contribution from large values of $k^r.$  Thus, when $Y\subset \AA^n$ is given by  \eqref{eq:green}, 
we will transform the problem into one that involves counting integral points of bounded size on  affine quadrics. Our bound needs to be uniform in the coefficients of the defining  polynomial and, since it may be of general interest,   we proceed to describe it here.
Let $q\in \ZZ[T_1,\ldots,T_\nu]$
be  a non-zero quadratic polynomial,
for $\nu\geq 2$. 
Let
$$
M(q;B)=\#\{\bt \in \ZZ^\nu: |\bt| \leq B, ~q(\bt)=0\},
$$
for any $B \geq 1$.  
We will require an upper bound for $M(q;B)$ which is uniform in the coefficients of $q$ and which 
is essentially as sharp and as general as possible. 
A trivial estimate is 
$M(q;B)=O_{\nu}(B^{\nu-1})$, which is optimal when $q$ is reducible over $\QQ$.
Assuming that $q$ is irreducible over $\QQ$, a result of Pila \cite{pila-ast} gives
$M(q;B)=O_{\ve,\nu}(B^{\nu-3/2+\ve})$, for any $\ve>0$. Again, this is essentially best possible,
as consideration of the polynomial 
$T_1-T_2^2$ shows. 
Let $q_0$ denote the quadratic part of $q$, so that $q_0=T_2^2$ in the previous example. 
One might hope for an improved bound when $q_0$  has rank at least  $2$. 
This is confirmed in the following result, which is a straightforward modification of ideas developed by Browning, Heath-Brown and Salberger \cite[\S\S 4--5]{pila}.

\begin{theorem}\label{thm:3}
Let $q\in \ZZ[T_1,\ldots,T_\nu]$ be quadratic, with $\nu\geq 2$. Let $\ve>0$.
Assume that $q$ is irreducible over $\QQ$ and that $\rank(q_0)\geq 2$. Then 
$$M(q;B)=O_{\ve,\nu}(B^{\nu-2+\ve}).$$ 
\end{theorem}

The implied constant in this result depends only on the choice of $\ve$ and the number $\nu$. This is the  most important  feature  of Theorem \ref{thm:3}, since 
it would be easy to prove a version of the theorem with an implied constant that is allowed to depend on $q$ by first diagonalising $q_0$ and then completing the square where possible. 

\begin{ack}
While working on this paper the authors were  
supported by ERC 
grants
\texttt{306457} and \texttt{239606}, respectively. 
\end{ack}

\section{Counting on symmetric varieties with congruences}\label{s:counting}

\subsection{The main estimate}
In this section we establish an asymptotic counting estimate 
for integral points on symmetric varieties that  satisfy a congruence condition. 
Let $$Y\simeq G/L\subset \mathbb{A}^n$$ be a 
symmetric variety satisfying the hypotheses (i)--(ii) from \S \ref{s:intro} and
\begin{enumerate}
	\item[(iii$^\prime$)] the group $G$ is $\QQ$-simple; and
	\item[(iv$^\prime$)] the group $G(\RR)$ is connected and has no compact factors.
\end{enumerate}
When $G$ is simply connected, $G(\RR)$ is connected, so that 
conditions (iii$^\prime$)--(iv$^\prime$) are weaker than conditions (iii)--(iv).


Recalling the definition \eqref{eq:cong} of $\Gamma_\ell$, 
our aim is to estimate the cardinality of the sets $$\{ \x \in \Gamma_\ell \y:\, |\x|\le H\},
$$
as $H\to\infty$, uniformly in $\ell$.

For $\y\in Y(\ZZ)$, we set 
$$
L_\y=\hbox{Stab}_G(\y)\quad\hbox{ and }\quad
B_H(\y)=\{ \x\in G(\RR) \y:\, |\x|\le H\}.
$$
We fix compatible volume forms $m_G$, $m_{L}$, $m_Y$ on $G(\RR)$, $L_\y(\RR)$, $G(\RR)\y$, respectively. Let 
$$
\mathcal{X}_\ell=G(\RR)/\Gamma_\ell\quad\hbox{ and }\quad \mathcal{Z}_{\y,\ell}=L_\y(\RR)/(\Gamma_\ell\cap L_\y(\RR)).
$$
We consider $\mathcal{Z}_{\y,\ell}$ as a submanifold of $\mathcal{X}_\ell$.
We denote by $m_{\mathcal{X}_\ell}$ and $m_{\mathcal{Z}_{\y,\ell}}$ the measures on $\mathcal{X}_\ell$ and $\mathcal{Z}_{\y,\ell}$ induced by the corresponding measure on $G(\RR)$ and $L_\y(\RR)$.
It follows from our assumptions that the spaces $\mathcal{X}_\ell$ and  $\mathcal{Z}_{\y,\ell}$ have finite measures.

With this notation, the main result of this section is the following.

\begin{theorem}\label{th:counting}
	Under assumptions (i)--(ii) and (iii$^\prime$)--(iv$^\prime$), there exists $\rho>0$ such that
	\begin{align*}
		|\Gamma_\ell \y\cap B_H(\y)|=\frac{m_{\mathcal{Z}_{\y,\ell}}(\mathcal{Z}_{\y,\ell})}{m_{\mathcal{X}_\ell}(\mathcal{X}_\ell)}\, m_Y(B_H(\y))
		+O(\ell^{\dim(L)+\dim(G)} m_Y(B_H(\y))^{1-\rho}).
	\end{align*}
	The implied constant in the error term is uniform over $\y\in Y(\ZZ)$ and $\ell\in \NN$.
\end{theorem}

The proof of Theorem \ref{th:counting} follows the strategy developed by Duke--Rudnick--Sarnak \cite{drs} and Eskin--McMullen \cite{em}. Quantitative estimates in this setting have also been obtained
by Benoist--Oh \cite{bo}.
The main novelty of our result is the uniformity over the congruence subgroups $\Gamma_\ell$,
which is pivotal for our application to power-free values of polynomials.

The following result shows that  Theorem \ref{th:counting} always provides a non-trivial estimate.

\begin{lemma}
	\label{l:inf}
	The space $G(\RR)\y\simeq G(\RR)/L_\y(\RR)$ is not-compact, and we have $m_Y(B_H(\y))\to \infty$
	as $H\to\infty$.
\end{lemma}

\begin{proof}
To simplify notation in this proof, we write $L$ for $L_\y(\RR)$.
Let $\theta$ be a Cartan involution of $G(\RR)$ that commutes with $\sigma$,
and $K$ is the corresponding maximal compact subgroup of $G(\RR)$.
Then we have the decompositions
$$
\hbox{Lie}(G(\RR))=\hbox{Lie}(K)\oplus \mathfrak{p}
\quad\hbox{and}\quad \hbox{Lie}(G(\RR))=\hbox{Lie}(L)\oplus \mathfrak{q}
$$
defined by the $(\pm 1)$-eigenspaces of $\theta$ and $\sigma$ respectively. By \cite[\S55]{berg}, $G(\RR)/L$ is diffeomorphic to the vector bundle $K\times_{K\cap L} (\mathfrak{p}\cap \mathfrak{q})$. In particular, the space 
$G(\RR)/L$ can only be compact if $\mathfrak{p}\cap \mathfrak{q}=0$.
We also have the Cartan decomposition 
$
G(\RR)=K\exp(\mathfrak{p}),
$
and its generalisation 
$$
G(\RR)=K\exp(\mathfrak{p}\cap \mathfrak{q})\exp(\mathfrak{p}\cap \hbox{Lie}(L))
$$
(see \cite[Prop.~2.2]{sch}), that define diffeomorphisms
$$
G(\RR)\simeq K\times \mathfrak{p}
\quad\hbox{and}\quad
G(\RR)\simeq K\times (\mathfrak{p}\cap \mathfrak{q})\times (\mathfrak{p}\cap \hbox{Lie}(L)).
$$
Hence, it follows that if $\mathfrak{p}\cap \mathfrak{q}=0$, then 
$\mathfrak{p}\subset \hbox{Lie}(L)$. Since 
$[\hbox{Lie}(K),\mathfrak{p}]\subset \mathfrak{p}$, the Lie algebra generated by $\mathfrak{p}$
is an ideal in $\hbox{Lie}(G(\RR))$ which corresponds to a connected normal subgroup
of $G(\RR)$ contained in $L$. Moreover, it is clear that this subgroup is cocompact.
Since $G(\RR)$ has no compact factors, it follows that $L=G(\RR)$, but the involution $\sigma$
has been assumed to be non-trivial. This contradiction shows that 
$\mathfrak{p}\cap \mathfrak{q}\ne 0$, and the space $G(\RR)/L$ is not compact.
The last assertion follows from explicit volume computations for symmetric spaces
(see \eqref{eq:volume} below).
\end{proof}

Theorem \ref{th:counting} is deduced from the following equidistribution result on the space $\mathcal{X}_\ell$, which will be established in \S \ref{s:corn}.
We denote by $\mu_{\mathcal{X}_\ell}$ and $\mu_{\mathcal{Z}_{\y,\ell}}$ the normalised measures on $\mathcal{X}_\ell$ and $\mathcal{Z}_{\y,\ell}$ respectively.

\begin{prop}\label{p:equid}
	Under assumptions (i)--(ii) and (iii$^\prime$)--(iv$^\prime$), there exist $q\ge 1$ and  $\rho_1>0$ such that for every $\phi\in C_c^\infty(\mathcal{X}_\ell)$, $\y\in Y(\ZZ)$, and $g\in G(\RR)$,
	$$
	\int_{\mathcal{Z}_{\y,\ell}} \phi(gz)\, d\mu_{\mathcal{Z}_{\y,\ell}}(z)=\int_{\mathcal{X}_\ell}\phi\, d\mu_{\mathcal{X}_\ell}+O_\y\left(m_{\mathcal{X}_\ell}(\mathcal{X}_\ell)|g \y|^{-\rho_1} \|\phi\|_{C^q}\right),
	$$
	where $\|\phi\|_{C^q}$ denotes the $C^q$-norm of the function $\phi$.
	The implied constant in the error term is uniform in $\ell$.
\end{prop}

We note that 
\begin{equation}
\label{eq:x_l}
m_{\mathcal{X}_\ell}(\mathcal{X}_\ell)\ll |\Gamma:\Gamma_\ell|\le G(\ZZ/\ell\ZZ)\ll \ell^{\dim(G)},
\end{equation}
in Proposition \ref{p:equid}.
Let us recall the definition of the $C^q$-norms.
These norms are defined with respect to 
a fixed basis $D_1,\ldots,D_n$ of the Lie algebra of $G(\RR)$.
For $\phi$ in $C_c^\infty(G(\RR))$
(or $C_c^\infty(\mathcal{X}_\ell)$) we set
\begin{equation}
	\label{eq:sobolev1}
	\|\phi\|_{C^q}=\sum_{D} \|D\phi\|_\infty,
\end{equation}
where the sum is taken over all monomials in $D_i$'s of degree at most $q$,
and $D_i$'s are right-invariant differential operators defined by
\begin{equation}
	\label{eq:sobolev2}
	D_i\phi(x)=\frac{d}{dt} \phi(\exp(tD_i)x)|_{t=0}.
\end{equation}

\begin{proof}[Proof of Theorem \ref{th:counting}]
	We note that by \cite{bh} the set $Y(\ZZ)$ consists of finitely many orbits of the arithmetic group $\Gamma$. 
	Therefore, it suffices to prove the claim of the theorem for $\y=\gamma_0 \y_0$, for a fixed $\y_0\in Y(\ZZ)$, with estimates which are uniform over $\gamma_0\in \Gamma$.
	Since  $B_H(\y)=B_H(\y_0)$
	in this case, in order to simplify notation, we denote this set by $B_H$ in subsequent computations. We also write $L$ for $L_{\y_0}$ and $\mathcal{Z}_\ell$
	for $\mathcal{Z}_{\y_0,\ell}$.
	
	Let 
	$$
	F_H(g)=\sum_{\gamma\in \Gamma_\ell/(\Gamma_\ell\cap L(\RR))} \chi_{B_H}(g\gamma \y_0).
	$$
	We note that this defines a function on $\mathcal{X}_\ell=G(\RR)/\Gamma_\ell$.
	Since $\Gamma_\ell$ is normal in $\Gamma$, we have
$		F_H(\gamma_0) =
		|\gamma_0\Gamma_\ell  \y_0\cap B_H|= |\Gamma_\ell \y\cap B_H|.
$
	For a real-valued $\phi\in C_c^\infty(\mathcal{X}_\ell)$, we 
	consider the inner product $\left< F_H,\phi \right>_{L^2(m_{\mathcal{X}_\ell})}$
	which can be unfolded as in \cite[p.~151]{drs}:
	\begin{align*}
		\left< F_H,\phi \right>_{L^2(m_{\mathcal{X}_\ell})}
		&=\int_{G(\RR)/\Gamma_\ell} \left(\sum_{\gamma\in \Gamma_\ell/(\Gamma_\ell\cap L(\RR))} \chi_{B_H}(g\gamma \y_0)\right) \phi(g\Gamma_\ell)\, dm_G(g)\\ 
		&=\int_{G(\RR)/(\Gamma_\ell\cap L(\RR))} \chi_{B_H}(g\y_0) \phi(g\Gamma_\ell)\, dm_G(g)\\ 
		&= \int_{G(\RR)/L(\RR)} \chi_{B_H}(g \y_0)
		\left( \int_{\mathcal{Z}_\ell} \phi(gz)\, dm_{\mathcal{Z}_\ell}(z)\right)\, dm_Y(g\y_0)\\
		&= \int_{B_H} 
		\left( \int_{\mathcal{Z}_\ell} \phi(gz)\, dm_{\mathcal{Z}_\ell}(z)\right)\, dm_Y(g\y_0)\\
		&= m_{\mathcal{Z}_\ell}(\mathcal{Z}_\ell) \int_{B_H} 
		\left( \int_{\mathcal{Z}_\ell} \phi(gz)\, d\mu_{\mathcal{Z}_\ell}(z)\right)\, dm_Y(g\y_0).
	\end{align*}
	Let $R\in (0,H)$. By Proposition \ref{p:equid}, when $|g\y_0|\ge R$, we have 
	$$
	\int_{\mathcal{Z}_\ell} \phi(gz)\, d\mu_{\mathcal{Z}_\ell}(z)=
	\int_{\mathcal{X}_\ell}\phi\, d\mu_{\mathcal{X}_\ell}+O(m_{\mathcal{X}_\ell}(\mathcal{X}_\ell)R^{-\rho_1} \|\phi\|_{C^q}).
	$$
	Also, it is clear that
	$$
	\int_{\mathcal{Z}_\ell} \phi(gz)\, d\mu_{\mathcal{Z}_\ell}(z)\ll \|\phi\|_\infty\le \|\phi\|_{C^q}.
	$$
	Hence, it follows that
$	\left< F_H,\phi \right>_{L^2(m_{\mathcal{X}_\ell})}$ is 
	\begin{align*}
		=~& m_{\mathcal{Z}_\ell}(\mathcal{Z}_\ell)\, m_Y(B_H\backslash B_R)\left(\int_{\mathcal{X}_\ell}\phi\, d\mu_{\mathcal{X}_\ell}
		+O(m_{\mathcal{X}_\ell}(\mathcal{X}_\ell)R^{-\rho_1} \|\phi\|_{C^q}) \right)\\
		& +O(m_{\mathcal{Z}_\ell}(\mathcal{Z}_\ell) m_Y(B_R) \|\phi\|_\infty )\\
		=~& m_{\mathcal{Z}_\ell}(\mathcal{Z}_\ell)\, m_Y(B_H)\left (\int_{\mathcal{X}_\ell}\phi\, d\mu_{\mathcal{X}_\ell}\right) \\
		&+O( m_{\mathcal{Z}_\ell}(\mathcal{Z}_\ell)(m_{\mathcal{X}_\ell}(\mathcal{X}_\ell)R^{-\rho_1}m_Y(B_H) +m_Y(B_R)) \|\phi\|_{C^q}) )\\
		=~& \frac{m_{\mathcal{Z}_\ell}(\mathcal{Z}_\ell)}{m_{\mathcal{X}_\ell}(\mathcal{X}_\ell)}\, m_Y(B_H)\left (\int_{\mathcal{X}_\ell}\phi\, dm_{\mathcal{X}_\ell}\right) \\
		&+O( m_{\mathcal{Z}_\ell}(\mathcal{Z}_\ell)(m_{\mathcal{X}_\ell}(\mathcal{X}_\ell) R^{-\rho_1}m_Y(B_H) + m_Y(B_R)) \|\phi\|_{C^q} ).
	\end{align*}
	It follows from \cite[Cor.~6.10]{gos} that
	\begin{equation}
	\label{eq:volume}
	m_Y(B_H)\sim v\, H^a(\log H)^b,\quad\hbox{ as $H\to\infty$},
	\end{equation}
	for some $v,a>0$ and $b\ge 0$. Hence, the last estimate with a suitable choice of the parameter $R$ implies that for some $\rho_2>0$
	\begin{equation}\begin{split}\label{eq:equid}
		\left< F_H,\phi \right>_{L^2(m_{\mathcal{X}_\ell})} =& \frac{m_{\mathcal{Z}_\ell}(\mathcal{Z}_\ell)}{m_{\mathcal{X}_\ell}(\mathcal{X}_\ell)}\, m_Y(B_H)\left (\int_{\mathcal{X}_\ell}\phi\, dm_{\mathcal{X}_\ell}\right)\\
		&+O(m_{\mathcal{Z}_\ell}(\mathcal{Z}_\ell) m_{\mathcal{X}_\ell}(\mathcal{X}_\ell) m_Y(B_H)^{1-\rho_2}\|\phi\|_{C^q}).
\end{split}
	\end{equation}
	
	We apply this estimate to a suitably chosen bump-function $\phi_\ve$ on $\mathcal{X}_\ell$. 
	We denote by $O^G_\ve$ the $\ve$-neighbourhood of identity 
	with respect to a Riemannian metric on $G(\RR)$.
	Let $\Phi_\ve$ be a smooth non-negative function supported on $O_\ve^G$ such that 
	$$
	\int_{G(\RR)} \Phi_\ve\, dm_G=1\quad\hbox{and}\quad \|\Phi_\ve\|_{C^q}\ll \ve^{-\beta},
	$$
	with some $\beta>0$ depending on $\dim(G)$.
	It follows from the definition of the $C^q$-norms
	(cf.\ \eqref{eq:sobolev1}--\eqref{eq:sobolev2}) that the functions $g\mapsto \Phi_\epsilon(gg_0)$, $g_0\in G$, have the same norms as $\Phi_\ve$.
	Let
	$$
	\phi_\ve(g\Gamma_\ell)=\sum_{\gamma\in\Gamma_\ell} \Phi_\ve(g\gamma_0^{-1}\gamma)=\sum_{\gamma\in\Gamma_\ell} \Phi_\ve(g\gamma\gamma_0^{-1}).
	$$
	This defines a function on $\mathcal{X}_\ell=G(\RR)/\Gamma_\ell$ which also satisfies 
	\begin{equation}
	\label{eq:f_e}
	\int_{\mathcal{X}_\ell} \phi_\ve\, dm_{\mathcal{X}_\ell}=1\quad\hbox{and}\quad \|\phi_\ve\|_{C^q}\ll \ve^{-\beta}.
	\end{equation}
	Our goal is to show that $|\Gamma_\ell \y \cap B_H|$
	can be approximated by the inner products $\left< F_H,\phi_\ve \right>_{L^2(m_{\mathcal{X}_\ell})}$.
	We observe that if for some $g\in G(\RR)$, we have 
	\begin{equation}
	\label{eq:ne_zero}
	\phi_\ve(g\Gamma_\ell)\ne 0,
	\end{equation}
	then $g\gamma_0^{-1}\in O^G_\ve \Gamma_\ell$
	and $g\in u \gamma_0 \Gamma_\ell$ for some $u\in O^G_\ve$. For such $g$,
	$$
	F_H(g)=|g\Gamma_\ell \y_0 \cap B_H|=|\gamma_0\Gamma_\ell \y_0 \cap u^{-1}B_H|.
	$$
	For $u\in O^G_\ve$, we have $u=e+O(\ve)$, so that there exists uniform $c>0$ such that
	$
	B_{(1-c\ve)H}\subset u^{-1}B_H\subset B_{(1+c\ve)H}.
	$
	Hence, we deduce that for $g$ satisfying \eqref{eq:ne_zero},
	$$
	|\gamma_0\Gamma_\ell \y_0 \cap B_{(1-c\ve)H}|\le |g\Gamma_\ell \y_0 \cap B_{H}|\le |\gamma_0\Gamma_\ell \y_0 \cap B_{(1+c\ve)H}|.
	$$
	This implies that $F_H(\gamma_0)=|\Gamma_\ell \y \cap B_H|$ satisfies
	$$
	F_{(1+c\ve)^{-1}H}(g) \le F_H(\gamma_0) \le F_{(1-c\ve)^{-1}H}(g).
	$$
	Hence, it follows from \eqref{eq:f_e} that 
	\begin{equation}
		\label{eq:inn}
		\left< F_{(1+c\ve)^{-1}H},\phi_\ve \right>_{L^2(m_{\mathcal{X}_\ell})} \le F_H(\gamma_0) \le \left< F_{(1-c\ve)^{-1}H},\phi_\ve \right>_{L^2(m_{\mathcal{X}_\ell})}.
	\end{equation}
	Applying \eqref{eq:equid}, we conclude that
	\begin{align*}
		F_H(\gamma_0) \le& 
		\frac{m_{\mathcal{Z}_\ell}(\mathcal{Z}_\ell)}{m_{\mathcal{X}_\ell}(\mathcal{X}_\ell)}\, m_Y(B_{(1-c\ve)^{-1} H})\\
		&+O\left(m_{\mathcal{Z}_\ell}(\mathcal{Z}_\ell)m_{\mathcal{X}_\ell}(\mathcal{X}_\ell) m_Y(B_{(1-c\ve)^{-1}H})^{1-\rho_2}\ve^{-\beta}\right).
	\end{align*}
	The volumes of the sets $B_H$ satisfy a regularity property. According to \cite[Appendix]{ems},  
	there exists $c_0>0$ such that for all $\delta\in (0,\delta_0)$ and $H\ge H_0$,
	$$
	m_Y(B_{(1+\delta)H})\le (1+c_0\delta) m_Y(B_{H}).
	$$
	Therefore, it follows that
	\begin{align*}
		F_H(\gamma_0)\le& \frac{m_{\mathcal{Z}_\ell}(\mathcal{Z}_\ell)}{m_{\mathcal{X}_\ell}(\mathcal{X}_\ell)}\,m_Y(B_{H})\\
		&+O\left( \frac{m_{\mathcal{Z}_\ell}(\mathcal{Z}_\ell)}{m_{\mathcal{X}_\ell}(\mathcal{X}_\ell)}\ve  m_Y(B_{H}) +  m_{\mathcal{Z}_\ell}(\mathcal{Z}_\ell) m_{\mathcal{X}_\ell}(\mathcal{X}_\ell) m_Y(B_{H})^{1-\rho_2}\ve^{-\beta}\right).
	\end{align*}
	We recall from \eqref{eq:x_l} that
	$m_{\mathcal{X}_\ell}(\mathcal{X}_\ell)\ll \ell^{\dim(G)}$, and similarly
	\begin{align*}
	m_{\mathcal{Z}_\ell}(\mathcal{Z}_\ell)=m_{L}(L(\RR)/(L(\RR)\cap \Gamma_\ell))&\ll |(\Gamma\cap L):(\Gamma_\ell\cap L)|\\
	&\le |\iota(L)(\ZZ/\ell \ZZ)|\\
	&\ll \ell^{\dim(L)}.
	\end{align*}
	Hence, we obtain
		\begin{align*}
		F_H(\gamma_0)\le& \frac{m_{\mathcal{Z}_\ell}(\mathcal{Z}_\ell)}{m_{\mathcal{X}_\ell}(\mathcal{X}_\ell)}\,m_Y(B_{H})\\
		&+O\left(\ell^{\dim(L)+\dim(G)} (\ve  m_Y(B_{H}) +   m_Y(B_{H})^{1-\rho_2}\ve^{-\beta})\right).
		\end{align*}
	Optimising in $\ve$, we deduce that there exists $\rho>0$ such that
	\begin{align*}
		|\Gamma_\ell \y \cap B_{H}|=F_H(\gamma_0)\le \frac{m_{\mathcal{Z}_\ell}(\mathcal{Z}_\ell)}{m_{\mathcal{X}_\ell}(\mathcal{X}_\ell)}\,m_Y(B_{H})+O(\ell^{\dim(L)+\dim(G)} m_Y(B_{H})^{1-\rho}).
	\end{align*}

	This proves the required upper bound on $|\Gamma_\ell \y \cap B_{H}|$.
	The lower estimate on $|\Gamma_\ell \y \cap B_{H}|$ is proved similarly using the lower bound from \eqref{eq:inn}. 
	This completes proof of  Theorem \ref{th:counting} assuming Proposition \ref{p:equid}.
\end{proof}

\subsection{Proof of the equidistribution result}\label{s:corn}

In this section we prove  Proposition \ref{p:equid}.
	To simplify notation, we write $L$ for $L_\y$ and $\mathcal{Z}_\ell$  for $\mathcal{Z}_{\y,\ell}$.
	We recall the Cartan decomposition 
	$$
	G(\RR)=K\,A\,L(\RR),
	$$
	where $K$ is a compact subgroup compatible with $L(\RR)$, and $A$ is a suitable Cartan subgroup complementary to $L(\RR)$ (see, for instance, \cite[Ch.7]{sch}). For $g\in G(\RR)$,
	we write $g=kah$ with $k\in K$, $a\in A$, and $h\in L(\RR)$. Then
	$$
	\int_{\mathcal{Z}_\ell}\phi(gz)\, d\mu_{\mathcal{Z}_\ell}(z)= \int_{\mathcal{Z}_\ell}\phi_k(az)\, d\mu_{\mathcal{Z}_\ell}(z),
	$$
	where $\phi_k\in C^\infty_c(\mathcal{X}_\ell)$ is given by $\phi_k(x)=\phi(kx)$. 
	Since $K$ is compact, 
	$$
	\|\phi_k\|_{C^q} \ll \|\phi\|_{C^q}\quad\hbox{and}\quad |gy|=|ka\y|\ll |a\y|
	$$
	uniformly for $k\in K$. Hence, the claim of the proposition will follow once we prove it for 
	$g=a\in A$. Moreover, without loss of generality we may assume that $a$
	belongs to a fixed positive Weyl chamber $A^+$ in $A$ for the action of $A$ on the Lie algebra of $G(\RR)$.
	
	In the proof we use parameters $R$, $\ve$, $\eta$ of the form 
	$$
	R=\eta_1 d(a,e),\quad \ve=e^{-\eta_2 d(a,e)},\quad \delta=e^{-\eta_3 d(a,e)},
	$$
	with some $\eta_1,\eta_2,\eta_3>0$ 	that will be specified later.
	We equip the space $\mathcal{X}_\ell$ with an invariant Riemannian metric
	induced from a right-invariant Riemannian metric on $G(\RR)$
	which is bi-invariant with respect to the maximal compact subgroup $K$.
	Fix $z_0\in \mathcal{Z}_\ell$ and set 
	$$
	\mathcal{Z}_{\ell,R}^-=\{z\in \mathcal{Z}_\ell:\, d(z, z_0)< R \}\quad
	\hbox{and}\quad
	\mathcal{Z}_{\ell,R}^+=\{z\in \mathcal{Z}_\ell:\, d(z, z_0)> R \}.
	$$
	It follows from \cite[Sec.~5]{km} that
	\begin{equation}
		\label{eq:cusp}
		m_{\mathcal{Z}_\ell}(\mathcal{Z}_{\ell,R}^+)\ll m_{\mathcal{Z}_\ell}(\mathcal{Z}_{\ell}) e^{-\theta R}\quad\hbox{and}\quad
		\mu_{\mathcal{Z}_\ell}(\mathcal{Z}_{\ell,R}^+)\ll e^{-\theta R},
	\end{equation}
	for some fixed $\theta>0$. 
	
	We refine the open cover $\mathcal{Z}_\ell=\mathcal{Z}_{\ell,R}^-\cup \mathcal{Z}_{\ell,R-1}^+$ further.
	Let $O^G_\ve$ and $O^L_\ve$ denote the $\ve$-neighbourhoods of identity in $G(\RR)$ and $L(\RR)$ respectively.
	Since these neighbourhoods are defined with respect to a invariant metric,
	$(O^G_\ve)^{-1}=O^G_\ve$ and $O^G_{\ve_1} O^G_{\ve_2}\subset O^G_{\ve_1+\ve_2}$,
	and similarly for $O_\ve ^L$.
	Let $\Omega=\{z_i:i\in I\}$ be a maximal subset of $\mathcal{Z}_{\ell,< R}$ such that
	$O^L_{\ve/2}z_i$, $i\in I$, are disjoint. We observe that then 
	\begin{equation}
		\label{eq:cover}
		\mathcal{Z}_{\ell,R}^-\subset \bigcup_{i\in I} O^L_{\ve}z_i.
	\end{equation}
	Indeed, suppose that
	$z\in \mathcal{Z}_{\ell,R}^-$, but $z$ does not belong to the union in \eqref{eq:cover}.
	Then the set $\Omega\cup\{z\}$ satisfies the same disjointness property as $\Omega$.
	Indeed, if  $O^L_{\ve/2}z \cap  O^L_{\ve/2}z_i\ne \emptyset$ for some $i$, then 
	$z\in (O^L_{\ve/2})^{-1}O^L_{\ve/2}z_i\subset O^L_{\ve}z_i$ which is not the case.
	This contradicts maximality of $\Omega$ and proves \eqref{eq:cover}.
	
	For future reference, we prove some basic properties of the set $\Omega$.
	First, we claim that for sufficiently small $\ve$, the map $O^G_{3\ve}\to O^G_{3\ve}z_i$ is injective.
	Let $z_i=h_i\Gamma_\ell$ with $h_i\in L(\RR)$ satisfying $d(h_i,e)<R$.
	If $u_1h_i\Gamma_\ell=u_2h_i\Gamma_\ell$ for some $u_1,u_2\in O^G_{3\ve}$, then for some $\gamma\in \Gamma_\ell$, we obtain 
	$$
	\gamma=h_i^{-1}u_2^{-1}u_1h_i\in h_i^{-1} O^G_{6\ve} h_i\subset O^G_{6e^{c R}\ve},
	$$
	with some fixed $c>0$. Let us choose $\ve\le \ve_0 e^{-cR}$ with sufficiently small $\ve_0>0$. Then it follows from discreteness of $\Gamma_\ell$ that $\gamma=e$, so that $u_1=u_2$.
	Hence, this shows that the map $O^G_{3\ve}\to O^G_{3\ve}z_i$ is injective.
	
	We will also need an upper bound on $|\Omega|$ which is easy to deduce from the disjointness
	property. Since the map $O^L_{\ve/2}\to O^L_{\ve/2}z_i$ is injective,
	$m_{\mathcal{Z}_\ell}(O^L_{\ve/2}z_i)\gg \ve^{d}$ where $d=\dim(L)$.
	This implies that 
	\begin{equation}
	\label{eq:omega}
	|\Omega|\ll m_{\mathcal{Z}_\ell}(\mathcal{Z}_\ell) \ve^{-d}. 
	\end{equation}		
	
	We choose a smooth function $\Psi$ on $L(\RR)$ such that 
	$$
	0\le \Psi\le 1,\;\;\; \hbox{supp}(\Psi)\subset O^L_{2\ve},\;\;\; \Psi=1\hbox{ on } O^L_{\ve},\;\;\; \|\Psi\|_{C^q}\ll \ve^{-\beta},
	$$
	for some $\beta>0$ depending on $\dim(L)$.
	It follows from the definition of $C^q$-norms
	(cf. \eqref{eq:sobolev1}--\eqref{eq:sobolev2}) that the family of functions $h\mapsto \Psi(hh_0)$, $h_0\in L(\RR)$, have the same Sobolev norms.
	Since the map $O^L_{2\ve}\to O^L_{2\ve}z_i$ is injective, we also obtain smooth functions
	$\chi_i$ on $\mathcal{Z}_\ell$ such that 
	$$
	\hbox{supp}(\chi_i)\subset O^L_{2\ve}z_i\quad\hbox{and}\quad
	\|\chi_i\|_{C^q}=\|\Psi\|_{C^q}\ll \ve^{-\beta}.
	$$
	We fix a total ordering on $I$ and set 
	$$
	\psi_i=\chi_i\prod_{j> i} (1-\chi_j)\quad\hbox{and}\quad
	\psi_\infty=1-\sum_i \psi_i=\prod_i(1-\chi_i).
	$$
	Then 
	$
	\hbox{supp}(\psi_i)\subset \hbox{supp}(\chi_i) \subset O^L_{2\ve}z_i,
	$
	and it follows from \eqref{eq:cover} that 
	$$
	\hbox{supp}(\psi_\infty)\subset \mathcal{Z}_{\ell,R-1}^+.
	$$
	It is also clear that 
	\begin{equation}\label{eq:psi}
		0\le \psi_i\le 1\quad\hbox{and}\quad\|\psi_i\|_{C^q}\ll \ve^{-\beta_1}, 
	\end{equation}
	with some fixed $\beta_1>0$.

	Let $P$ be the non-expanding horospherical subgroup of $G(\RR)$ corresponding to $A^+$.
	This is the connected Lie subgroup of $G(\RR)$ whose Lie algebra consists of $X$
	such that $\frac{\| \hbox{\small Ad}(a) X \|}{\|X\|}$ is uniformly bounded as $a\in A^+$.
	This property implies, in particular, that for $p$ in a neighbourhood of identity in $P$ and all $a\in A^+$,
	\begin{equation}
		\label{eq:dist}
		d(apa^{-1},e)\ll d(p,e).
	\end{equation}
	We note that since $L(\RR)$ is a symmetric subgroup in $G(\RR)$, it follows that
	$$
	\hbox{Lie}(G(\RR))=\hbox{Lie}(P)+\hbox{Lie}(L(\RR))
	$$
	(see, for instance, \cite[p.~199]{em}).
	In particular, there is a subspace $V$ of $\hbox{Lie}(P)$
	such that 
	\begin{equation}
		\label{eq:lie}
		\hbox{Lie}(G(\RR))=V\oplus \hbox{Lie}(L(\RR)).
	\end{equation}
	Let $O_\ve^V$ denote the $\ve$-neighbourhood of identity in $\exp(V)$.
	It follows from \eqref{eq:lie} that the product map 
	\begin{equation}
		\label{eq:prod}
		\exp(V)\times L(\RR)\to G(\RR)
	\end{equation}
	is a diffeomorphism in
	a neighbourhood of identity. For $g$ in a neighbourhood of identity in $G(\RR)$,
	we write $g={\bf v}(g){\bf h}(g)$ where ${\bf v}$ and ${\bf h}$ are the smooth maps
	realising this diffeomorphism.
	
	We observe that with respect to the decomposition \eqref{eq:prod}, the Haar measure $m_G$ restricted to $O_\ve^V O_{2\ve}^L$
	decomposes as a product $m_V\otimes m'_L$ where $m_V$ is a smooth measure on $O_\ve^V$,
	and $m'_L$ is the restriction of $m_L$ to $O_{2\ve}^L$.
	We also note that under the map $h\mapsto h z_i$, the measure $m'_L$ projects to 
	$m_{\mathcal{Z}_\ell}|_{O_{2\ve}^L z_i}$.
	We choose a smooth non-negative function $\sigma$ on $\exp(V)$ such that 
	\begin{equation}
		\label{eq:sigma}
		\int_{\exp(V)} \sigma\, dm_V=1,\;\;\; \hbox{supp}(\sigma)\subset O_\delta^V,\;\;\; \|\sigma\|_{C^q}\ll \delta^{-\beta_2},
	\end{equation}
	with some $\beta_2>0$ depending on $\dim(V)$.
	Then
	\begin{align*}
		\int_{\mathcal{Z}_\ell} \phi(az)\psi_i (z)\, dm_{\mathcal{Z}_\ell}(z)
		&=\int_{\exp(V)\times O_{2\ve}^L} \phi(ah z_i) \sigma(v)\psi_i (hz_i)\, dm_V(v)dm'_{L}(h)\\
		&=\int_{\exp(V)\times O_{2\ve}^L} \phi(ah z_i) \Phi_i(vh)\, dm_V(v)dm'_{L}(h),
	\end{align*}
	where $\Phi_i$ is the smooth function supported on $O_\delta^V O_{2\ve}^L$, with $\delta\le\ve$, defined by
	$$
	\Phi_i(g)=\sigma({\bf v}(g))\tilde \psi_i ({\bf h}(g)),
	$$
	where $\tilde \psi_i (h) = \psi_i (h z_i)$. Then $\|\tilde \psi_i\|_{C^q}=\|\psi_i\|_{C^q}$, and
	it follows from \eqref{eq:psi} and \eqref{eq:sigma} that
	$$
	\| \Phi_i\|_{C^q}\ll \delta^{-\beta_3},
	$$
	with some $\beta_3>0$. By \eqref{eq:dist}, for $v\in O_\delta^V$, 
	\begin{align*}
		|\phi(avz)-\phi(az)|\ll \|\phi\|_{C^1}\, d(ava^{-1},e)\ll \|\phi\|_{C^1} \delta.
	\end{align*}
	Since
	\begin{align*}
	\left|\int_{\exp(V)\times O_\ve^L} \Phi_i(vh)\, dm_V(v)dm'_{L}(h)\right|
	&=
	\left( \int_{\exp(V)} \sigma\, dm_V\right) \left(\int_{O_\ve^L} \tilde \psi_i(h)\,dm'_{L}(h)\right)\\ &\ll \ve^d,
	\end{align*}
	where $d=\dim(L)$,
	we obtain that
	\begin{align*}
		\int_{\mathcal{Z}_\ell} \phi(az)\psi_i (z)\, dm_{\mathcal{Z}_\ell}(z)
		=&\int_{\exp(V)\times O_{2\ve}^L} \phi(avhz_i) \Phi_i(vh)\, dm_V(v)dm'_{L}(h)\\
		&+O(\|\phi\|_{C^1} \delta \ve^d)\\
		=&\int_{O_\delta^V O_{2\ve}^L} \phi(agz_i) \Phi_i(g)\, dm_G(g)
		+O(\|\phi\|_{C^1} \delta\ve^d).
	\end{align*}
	We recall that the map $g\mapsto g z_i$ is injective on $O_{3\ve}^G$.
	Hence, under this map 
	the measure $m_G|_{O_{3\ve}^G}$ projects to the measure $m_{\mathcal{X}_\ell}|_{O_{3\ve}^G z_i}$, and since $O_\delta ^V O_{2\ve}^L\subset O_{3\ve}^G$,
	$\Phi_i$ defines a function $\phi_i$ supported on $O_{3\ve}^G z_i$ such that
	\begin{equation}
	\label{eq:phi_ccc}
	\|\phi_i\|_{C^q}=\|\Phi_i\|_{C^q}\ll \delta^{-\beta_3},
	\end{equation}
	 and 
	\begin{align*}
		\int_{O_{3\ve}^G} \phi(agz_i) \phi_i(gz_i)\, dm_G(g) =
		\int_{\mathcal{X}_\ell} \phi(ax) \phi_i(x)\, dm_{\mathcal{X}_\ell}(x). 
	\end{align*}
	Combining the above estimates, we conclude that
$$
		\int_{\mathcal{Z}_\ell} \phi(az)\phi_i (z)\, dm_{\mathcal{Z}_\ell}(z)
		=\int_{\mathcal{X}_\ell} \phi(ax) \phi_i(x)\, dm_{\mathcal{X}_\ell}(x)+O(\|\phi\|_{C^1} \delta\ve^d).
$$
	
	This formula allows us to use the exponential decay property of matrix coefficients
	for representations of $G(\RR)$ to estimate the original integral.
	It is known from the works \cite{bs,c}, which established bounds towards the generalised 
	Ramanujan conjectures, that the action of each simple factor of $G(\RR)$
	on the congruence quotients $\mathcal{X}_\ell=G(\RR)/\Gamma_\ell$ has the uniform spectral gap property. 
	Namely, the unitary representation of non-compact simple factors of $G(\RR)$ 
	on the orthogonal complement of the constant functions in $L^2(\mathcal{X}_\ell)$
	are uniformly isolated from the trivial representations.
	Then, by \cite[\S3.4]{km}, there exists $\rho>0$ such that 
	\begin{equation}\begin{split}\label{eq:exp_mix}
		\int_{\mathcal{X}_\ell} \phi(ax) \phi_i(x)\, d\mu_{\mathcal{X}_\ell}(x)
		=&
		\left(\int_{\mathcal{X}_\ell} \phi\, d\mu_{\mathcal{X}_\ell}\right)
		\left(\int_{\mathcal{X}_\ell} \phi_i \, d\mu_{\mathcal{X}_\ell}\right)\\
		&+O\left(e^{-\rho d(a,e)}\|\phi\|_{C^q} \|\phi_i\|_{C^q}\right).
	\end{split}\end{equation}
	We note that the exponent $\rho$ is determined by the isolation property of 
	the unitary representation, so that it is independent of $\ell$.

	Since $\sum_i \psi_i +\psi_\infty=1$, we obtain
	\begin{align*}
		\int_{\mathcal{Z}_\ell} \phi(az)\, dm_{\mathcal{Z}_\ell}(z)
		=\sum_i \int_{\mathcal{Z}_\ell} \phi(az)\psi_i(z)\, dm_{\mathcal{Z}_\ell}(z)
		+\int_{\mathcal{Z}_\ell} \phi(az)\psi_\infty(z)\, dm_{\mathcal{Z}_\ell}(z).
	\end{align*}
	The last term can be estimated using the fact that 
	$\hbox{supp}(\psi_\infty)\subset \mathcal{Z}_{\ell,R-1}^+$. This gives
	$$
	\int_{\mathcal{Z}_\ell} \phi(az)\psi_\infty(z)\, dm_{\mathcal{Z}_\ell}(z)\le  m_{\mathcal{Z}_\ell}(\mathcal{Z}_{\ell,R-1}^+) \|\phi\|_\infty
	\ll m_{\mathcal{Z}_\ell}(\mathcal{Z}_\ell) e^{-\theta R} \|\phi\|_\infty,
	$$
	whence
	\begin{align*}
		\int_{\mathcal{Z}_\ell} \phi(az)\, dm_{\mathcal{Z}_\ell}(z)
		=&\sum_i\int_{\mathcal{X}_\ell} \phi(ax) \phi_i(x)\, dm_{\mathcal{X}_\ell}(x)\\
		&+ O\left(|\Omega|\|\phi\|_{C^1}\delta\ve^d + e^{-\theta R} \|\phi\|_\infty\right)\\
		=&m_{\mathcal{X}_\ell}(\mathcal{X}_\ell) \sum_i\int_{\mathcal{X}_\ell} \phi(ax) \phi_i(x)\, d\mu_{\mathcal{X}_\ell}(x)\\
		&+ O\left(|\Omega|\|\phi\|_{C^1}\delta\ve^d + m_{\mathcal{Z}_\ell}(\mathcal{Z}_\ell)e^{-\theta R} \|\phi\|_\infty\right).
	\end{align*}
	Next, we apply \eqref{eq:exp_mix}, combined with estimates \eqref{eq:phi_ccc} and \eqref{eq:omega}, to deduce that
the right hand side is
	\begin{align*}
		= ~& m_{\mathcal{X}_\ell}(\mathcal{X}_\ell) \sum_i
		\left(\int_{\mathcal{X}_\ell} \phi\, d\mu_{\mathcal{X}_\ell}\right)
		\left( \int_{\mathcal{X}_\ell} \phi_i \, d\mu_{\mathcal{X}_\ell}\right)\\
		&+O\left(m_{\mathcal{X}_\ell}(\mathcal{X}_\ell) |\Omega|  e^{-\rho d(a,e)}\|\phi\|_{C^q}\delta^{-\beta_3}+ |\Omega|\|\phi\|_{C^1}\delta\ve^d + m_{\mathcal{Z}_\ell}(\mathcal{Z}_\ell)e^{-\theta R} \|\phi\|_\infty\right)\\
		=~& 
		\left(\int_{\mathcal{X}_\ell} \phi\, d\mu_{\mathcal{X}_\ell}\right)
		\left(\sum_i \int_{\mathcal{X}_\ell} \phi_i \, dm_{\mathcal{X}_\ell}\right)\\
		&+O\left( m_{\mathcal{Z}_\ell}(\mathcal{Z}_\ell) (m_{\mathcal{X}_\ell}(\mathcal{X}_\ell) \ve^{-d} \delta^{-\beta_3} e^{-\rho d(a,e)}+  \delta + e^{-\theta R} )\|\phi\|_{C^q}\right).
	\end{align*}
	Here we used that $m_{\mathcal{X}_\ell}=m_{\mathcal{X}_\ell}(\mathcal{X}_\ell)\mu_{\mathcal{X}_\ell}$.
	Using \eqref{eq:sigma}, the sum above is 
	\begin{align*}
		\sum_i \int_{\mathcal{X}_\ell} \phi_i \, dm_{\mathcal{X}_\ell}
		&=\sum_i\int_{O_{3\ve}^G} \phi_i(gz_i)\, dm_G(g)\\&=
		\sum_i\int_{O_{3\ve}^G} \Phi_i(g)\, dm_G(g)\\
		&=\sum_i\int_{O_\delta^V\times O_{2\ve}^L} \sigma(v)\psi_i(hz_i)\, dm_V(v)dm'_L(h)\\
		&=\sum_i \int_{\mathcal{Z}_\ell} \psi_i\, dm_{\mathcal{Z}_\ell}.
		\end{align*}
Hence, it follows from  \eqref{eq:cusp} that 		
	\begin{align*}
		\sum_i \int_{\mathcal{X}_\ell} \phi_i \, dm_{\mathcal{X}_\ell}
		&=m_{\mathcal{Z}_\ell}(\mathcal{Z}_\ell) \int_{\mathcal{Z}_\ell} (1-\psi_\infty)\, d\mu_{\mathcal{Z}_\ell}\\
		&= m_{\mathcal{Z}_\ell}(\mathcal{Z}_\ell) \left(1+O\left( e^{-\theta R} \right)\right).
	\end{align*}
	Since $m_{\mathcal{Z}_\ell}=m_{\mathcal{Z}_\ell}(\mathcal{Z}_\ell)\mu_{\mathcal{Z}_\ell}$, we conclude that
	\begin{align*}
		\int_{\mathcal{Z}_\ell} \phi(az)\, d\mu_{\mathcal{Z}_\ell}(z)
		=& \int_{\mathcal{X}_\ell} \phi\, d\mu_{\mathcal{X}_\ell}\\
		&+O\left( (m_{\mathcal{X}_\ell}(\mathcal{X}_\ell)  \ve^{-d}\delta^{-\beta_3} e^{-\rho d(a,e)}+ \delta + e^{-\theta R} )\|\phi\|_{C^q}\right).
	\end{align*}
	We recall that this estimate holds under the previously made assumptions:
	$$
	\ve\le \ve_0 e^{-cR}\quad\hbox{and}\quad \delta\le\ve.
	$$
	We take $\ve= \ve_0 e^{-cR}$ and $\delta= \ve_0e^{dcR/(\beta_3+1)} e^{-\rho d(a,e)/(\beta_3+1)}$.
	This gives 
	\begin{align*}
	\int_{\mathcal{Z}_\ell} \phi(az)\, d\mu_{\mathcal{Z}_\ell}(z)
	=& \int_{\mathcal{X}_\ell} \phi\, d\mu_{\mathcal{X}_\ell}\\
	&+O\left( m_{\mathcal{X}_\ell}(\mathcal{X}_\ell) ( e^{dcR/(\beta_3+1)} e^{-\rho d(a,e)/(\beta_3+1)}+ e^{-\theta R} )\|\phi\|_{C^q}\right).
	\end{align*}
	We choose $R=\eta d(a,e)$ with sufficiently small $\eta>0$. Then $\delta\le \ve$, and 
	 we deduce that 
	\begin{align}\label{eq:final}
		\int_{\mathcal{Z}_\ell} \phi(az)\, d\mu_{\mathcal{Z}_\ell}(z)
		= \int_{\mathcal{X}_\ell} \phi\, d\mu_{\mathcal{X}_\ell}
		+O\left( m_{\mathcal{X}_\ell}(\mathcal{X}_\ell) e^{-\rho' d(a,e)} \|\phi\|_{C^q}\right),
	\end{align}
	for some $\rho'>0$. 

	To finish the proof, it remains to compare $d(a,e)$ and $\|a\y\|$.
	First, we note that since $d$ is an invariant Riemannian metric,
	$$
	\|\log a\| \ll d(a,e)\ll \|\log a\|,
	$$
	where $\|\cdot\|$ denotes the norm induced by the Riemannian metric on $\hbox{Lie}(A)$.
	Since the $A$-action on $\RR^n$ is diagonalisable,
	we can write $\y=\sum_i \y_i$ where $\y_i$'s are linearly
	independent eigenvectors of $A$. Then 
	$$
	a\y=\sum_i e^{\lambda_i(\log a)}\y_i
	$$
	for some characters $\lambda_i$ on $\hbox{Lie}(A)$, and
	$$
	|a\y|\ll \exp\left(\max_i \lambda_i(\log a)\right)\le \exp (c \|\log a\|)
	$$
	for some fixed $c>0$. Hence, it follows from \eqref{eq:final} that
	\begin{align*}
		\int_{\mathcal{Z}_\ell} \phi(az)\, d\mu_{\mathcal{Z}_\ell}(z)
		= \int_{\mathcal{X}_\ell} \phi\, d\mu_{\mathcal{X}_\ell}
		+O\left( m_{\mathcal{X}_\ell}(\mathcal{X}_\ell) |a\y|^{-\rho''} \|\phi\|_{C^q}\right),
	\end{align*}
	with $\rho''=c\rho'>0$. This therefore completes the proof of the Proposition \ref{p:equid}.

\begin{rem}\label{r:22}
We proved Theorem \ref{th:counting} under the assumption that the group $G$ is $\QQ$-simple,
but the method of the proof sometimes works without this assumption.
The only place where this assumption was used is the exponential mixing estimate
\eqref{eq:exp_mix}. When $G$ is not $\QQ$-simple, the space $\mathcal{X}_\ell$
has a finite cover $\prod_{i=1}^r \mathcal{X}^{(i)}_\ell$, where $\mathcal{X}^{(i)}_\ell$
are the spaces corresponding to $\QQ$-simple factors of $G$. In this case, 
we can generalise  \eqref{eq:exp_mix} to give
	\begin{align*}
	\int_{\mathcal{X}_\ell} \phi(ax) \psi_i(x)\, d\mu_{\mathcal{X}_\ell}(x)
	=~&
	\left(\int_{\mathcal{X}_\ell} \phi\, d\mu_{\mathcal{X}_\ell}\right)
	\left(\int_{\mathcal{X}_\ell} \psi_i \, d\mu_{\mathcal{X}_\ell}\right)\\
	&+O\left(e^{-\rho D(a)}\|\phi\|_{C^q} \|\psi_i\|_{C^q}\right),
\end{align*}
where $D(a)=\min_i d(a^{(i)},e)$, and $a=a^{(1)}\cdots a^{(r)}$ is the decomposition of $a$
with respect to the $\QQ$-simple factors. Hence, if one shows that for every $a\in A$,
\begin{equation}
\label{eq:D}
D(a)\gg d(a,e),
\end{equation}
then the proof of Theorem \ref{th:counting} can be completed exactly as before.

We are particularly interested in quadric hypersurfaces $\{Q=m\}$ of signature $(2,2)$.
After a suitable real change of variable, this quadratic surface can be reduced to the form
$X_1X_2-X_3X_4=m$ with $m>0$. Then after identifying $\RR^4$ with the space $\hbox{M}_2(\RR)$ of matrices, $Q$ will be given by the determinant, and $G(\RR)\simeq \hbox{SL}_2(\RR)\times \hbox{SL}_2(\RR)$ with the action given 
$$
(g_1, g_2)\cdot X \mapsto g_1Xg_2^{-1}, \quad\quad (g_1,g_2)\in G(\RR),\;\; X\in \hbox{M}_2(\RR).
$$
Then $Y(\RR)\simeq G(\RR)/L(\RR)$, where $L(\RR)$ is the diagonal subgroup of $G(\RR)$.
It is the symmetric subgroup with respect to the involution 
$(g_1, g_2)\mapsto (g_2, g_1)$. In this case the Cartan subgroup complementary to $L(\RR)$
is 
$$
A=\{(b,b^{-1}):b\in B\},
$$
where $B$ denotes the diagonal subgroup of $\hbox{SL}_2(\RR)$.
It is clear that \eqref{eq:D} holds in this case,  so that Theorem \ref{th:counting} holds as well.
\end{rem}

\subsection{Consequences}

Our next goal is to estimate 
$$
\#\{\x\in Y(\ZZ)\cap \mathcal{O}_{\A}:\, |\x|\leq H, ~\x\equiv \bxi \bmod{\ell}\}
$$
for a given orbit $\mathcal{O}_{\A}$ of $G(\A)$ in $Y(\A)$.
To state this result we use a cohomological invariant $\delta:Y(\A)\to \{0,|\hbox{Pic}(L)|\}$ 
introduced by Borovoi and Rudnick in \cite{borovoi}.
This invariant is constant on orbits $\mathcal{O}_{\A}$ of $G(\A)$ in $Y(\A)$
and has the property that 
$$
\delta(\mathcal{O}_{\A})=0\quad \Longleftrightarrow \quad Y(\QQ)\cap \mathcal{O}_{\A}=\emptyset.
$$
We note that $\delta\equiv 1$  when $L$ is semisimple and simply connected. 
In particular, $\delta\equiv 1$ in the case of quadric hypersurfaces with $n\ge 4$.

Let $\mathcal{O}_{\A}$ be an orbit of $G(\A)$ in $Y(\A)$.
This orbit is of the form $\mathcal{O}_{\A}=\mathcal{O}_{\infty}\times\mathcal{O}_{f}$
where $\mathcal{O}_\infty$ is an orbit of $G(\RR)$ in $Y(\RR)$,
and $\mathcal{O}_f$ is an orbit of $G(\A_f)$ in $Y(\A_f)$.
For $\ell\in \NN$ and $\bxi\in Y(\ZZ/\ell\ZZ)$,
we consider a family of open subsets 
$B_f(\bxi,\ell)$ of $Y({\A}_f)$ defined by 
\begin{align*}
B_f(\bxi,\ell)&=\prod_{p<\infty} B_p(\bxi,\ell),
\end{align*}
where
$$
B_p(\bxi,\ell)=\left\{\y\in  Y(\ZZ_p):\, \y \equiv \bxi \bmod{p^{v_p(\ell)}} \right\}.
$$
We also set 
$$
\mathcal{O}_\infty(H)=\{\y\in \mathcal{O}_\infty:\, |\y|\le H\}.
$$

We fix a {\em gauge form} on $Y$; i.e., a nowhere zero regular differential form of top degree.
Since $Y$ is a homogeneous variety of a semisimple group, such a form exists and is
unique up to a scalar multiple. This defines a measure $m_Y$ on $Y(\RR)$,
induced by the gauge form, and 
a measure $m_{Y,f}$ on $Y(\A_f)$.
We refer to \cite[\S1]{borovoi} for a detailed discussion of gauge forms and corresponding measures. When $L$ is semisimple, the measure $m_{Y,f}$ is the product of 
measures $m_{Y,p}$ on $Y(\QQ_p)$ induced by the gauge form
$$
m_{Y,f}= \prod_{p<\infty} m_{Y,p}.
$$
To define $m_{Y,f}$ in general, we need to introduce suitable convergence factors.
Let $\rho_L$ denote the representation of $\hbox{Gal}(\bar \QQ/\QQ)$ on the space $X^*(L)\otimes \QQ$ of characters of $G$,  let $t_L$ be the rank of the group of $\QQ$-characters of $G$, and let
\begin{equation}\label{eq:artin}
L(s,\rho_L)= \prod_{p<\infty} L_p(s,\rho_L)
\end{equation}
be the Artin $L$-function associated to $\rho_L$. We recall that $L(s,\rho_L)$ has a pole of order $t_L$ at $s=1$. The measure $m_{Y,f}$ on $Y(\A_f)$ is defined by
$$
m_{Y,f}= \left(\lim_{s\to 1} (s-1)^{t_L}L(s,\rho_L)\right)\prod_{p<\infty} L_p(1,\rho_L)^{-1} m_{Y,p}.
$$
Assumption (ii) implies that $t_L=0$, so that we also have 
$$
m_{Y,f}= \prod_{p<\infty} m_{Y,p},
$$
but this convergence is only conditional.

For our  next results, we henceforth assume that $Y\simeq G/L$ satisfies the  assumptions  (i)--(iv).

\begin{corollary}\label{cor:orbit} 
	Under assumptions (i)--(iv), there exists $\rho>0$ such that
	for every orbit $\mathcal{O}_{\A}$ of $G(\A)$ in $Y(\A)$, we have
	\begin{align*}
	\#\{\x\in Y(\ZZ)\cap \mathcal{O}_{\A}:\,  |\x|\leq H, ~&\x\equiv \bxi \bmod{\ell}\}\\
	=~&
	\delta(\mathcal{O}_{\A})m_Y(\mathcal{O}_\infty(H))m_{Y,f}(\mathcal{O}_{f}\cap B_f(\bxi,\ell))\\ &+
	O\left(\ell^{\dim(L)+2\dim(G)} m_Y(\mathcal{O}_\infty(H))^{1-\rho}\right).
	\end{align*}
	The implied constant in the error term is uniform 
	 in $\ell$ and $\bxi$.
	 \end{corollary}

\begin{proof}
	We introduce compact open subgroups
	$$
	K_f(\ell)=\left\{g\in \prod_{p<\infty} G(\ZZ_p):\, \iota(g_p) \equiv id \bmod{p^{v_p(\ell)}}\right\}
	$$
	of $G(\A_f)$ and set 
	$$
	K(\ell)=G(\RR)\times K_f(\ell).
	$$ 
	Then $\Gamma_\ell=G(\QQ)\cap K(\ell)$ are precisely the congruence subgroups defined in \eqref{eq:cong}. The group $\Gamma_\ell$ acts on $Y(\QQ)\cap B(\bxi,\ell)$.
	Given an orbit $\mathcal{O}$ of $\Gamma_\ell$ in $Y(\QQ)\cap B(\bxi,\ell)$,
	following \cite[\S4]{borovoi} we define its {\em weight} $w(\mathcal{O})$ as follows.
	For $\y\in\mathcal{O}$, we fix gauge forms on $G$ and $L_\y$ 
	which are compatible with the chosen gauge form on $Y$.
	These forms define the corresponding measures $m_G$ on $G(\RR)$ and $m_{L_\y}$ on $L_\y(\RR)$
	which also induce measures 
	$m_{\mathcal{X}_\ell}$ and $m_{\mathcal{Z}_{\y,\ell}}$ on 
	$\mathcal{X}_\ell=G(\RR)/\Gamma_\ell$ and $\mathcal{Z}_{\y,\ell}=L_\y(\RR)/(\Gamma_\ell\cap L_\y(\RR))$.
	The weight of the orbit $\mathcal{O}$ is defined by 
	$$
	w(\mathcal{O})=\frac{m_{\mathcal{Z}_{\y,\ell}}(\mathcal{Z}_{\y,\ell})}{m_{\mathcal{X}_\ell}(\mathcal{X}_\ell)}.
	$$
	One can check that this definition is independent of the choice of $\y$.
	Using the new notation, Theorem \ref{th:counting} can be restated as follows:
	\begin{equation}
	\label{eq:orbits}
	|\mathcal{O}\cap \mathcal{O}_\infty(H)| =w(\mathcal{O}) m_Y(\mathcal{O}_\infty(H))
	+O(\ell^{\dim(L)+\dim(G)} m_Y(\mathcal{O}_\infty(H))^{1-\rho}).
	\end{equation}
	Let $B=\mathcal{O}_\infty\times (\mathcal{O}_{f}\cap B_f(\bxi,\ell))$.
	Then 
	$$
	Y(\QQ)\cap B=\{\x\in Y(\ZZ)\cap \mathcal{O}_{\A}:\, \x\equiv \bxi \bmod{\ell}\}.
	$$
	We note that $L$ is reductive by \cite[Thm.~3.5]{bh} and, in particular,
	unimodular. Since $G$ is simply connected and $L$
	is symmetric, it follows from \cite[\S 8]{stein} that $L$ is connected. 
	Hence, all the assumptions of \cite[\S4]{borovoi} are satisfied, and
	according to \cite[Thm.~4.2]{borovoi},
	$$
	\sum_{\mathcal{O}\subset B} w(\mathcal{O})=\delta(\mathcal{O}_{\A})m_{Y,f}(\mathcal{O}_{f}\cap B_f(\bxi,\ell)),
	$$
	where the sum is taken over the orbits $\mathcal{O}$ of $\Gamma_\ell$ contained in $B$.
	Hence, summing \eqref{eq:orbits} over these orbits, we deduce the corollary.
	We note that the number orbits is at most $O(|\Gamma:\Gamma_\ell|)=O(\ell^{\dim(G)})$
	which contributes an additional factor to the error term.
\end{proof}

Finally, we deduce an estimate for the number of 
$\x\in Y(\ZZ)$ with $|\x|\leq H$ and $\x\equiv \bxi \bmod{\ell}.$
Let 
$$
Y(\RR)_H=\{\x\in Y(\RR): |\x|\le H\}.
$$
Since the set $Y(\ZZ)$ consist of finitely many  orbits of $\Gamma$, by \cite{bh}, 
we can sum the estimates from Corollary \ref{cor:orbit} to conclude as follows.

\begin{corollary}\label{cor:variety} 
	There exists $\rho>0$ such that
	\begin{align*}
	\#\{\x\in Y(\ZZ):&\, |\x|\leq H, ~\x\equiv \bxi \bmod{\ell}\}\\
	=~&
	\int_{Y(\RR)_H\times B_f(\bxi,\ell)} \delta \, d(m_Y\otimes m_{Y,f})
	+ O\left(\ell^{\dim(L)+2\dim(G)} V(H)^{1-\rho}\right),
	\end{align*}
	where 
	$$
	V(H)=\max_{\mathcal{O}_\infty\subset Y(\RR)} m_Y(\mathcal{O}_\infty(H))
	$$
	and where $\mathcal{O}_\infty$ runs over orbits of $G(\RR)$ in $Y(\RR)$.
	The implied constant in the error term is uniform on $\bxi$ and $\ell$.
	If, in addition, $L$ is assumed to be semisimple and simply connected, then
	\begin{align*}
	\int_{Y(\RR)_H\times B_f(\bxi,\ell)} \delta \, d(m_Y\otimes m_{Y,f})=
	m_Y(Y(\RR)_H)m_{Y,f}(B_f(\bxi,\ell)).
	\end{align*}
\end{corollary}

\section{Small moduli and the proof of Theorems \ref{t:symmetric} and \ref{t:symmetric2}}\label{s:small}

\subsection{Small moduli}

Let $Y\simeq G/L \subset \AA^n$ be a symmetric variety satisfying (i)--(iv)
and let  $r\geq 2$. In \S\S 3.1--3.2 we additionally assume that $L$ is semisimple and simply connected
and that the smoothness assumption \eqref{eq:smooth} holds.
The case when $L$ is not a semisimple simply connected group will be discussed in \S 3.3. 

We recall the expression \eqref{eq:spud} 
for $N_r(Y,f;H)$.
The goal of this section is to estimate the contribution from small moduli
\begin{equation}
\label{eq:n1}
N^{(1)}(H)
=
\hspace{-0.1cm}
\sum_{k\leq H^\Delta} 
\hspace{-0.2cm}
\mu(k) 
\card \{\x\in Y(\ZZ): |\x|\leq H, ~ 0\neq f(\x)\equiv 0\bmod{k^r}
\},
\end{equation}
as $H\to\infty$,
for given   $\Delta>0$.  
We shall need to separate  the contribution from $\x$ such that $f(\x)=0$. Accordingly, we write
\begin{align*}
N^{(1)}(H)
=~&
\sum_{k\leq H^\Delta} 
\mu(k) 
\card \{\x\in Y(\ZZ): |\x|\leq H, ~ f(\x)\equiv 0\bmod{k^r}
\}\\ &+ O\left(H^{\Delta}E(f;H)\right),
\end{align*}
where for any $g\in \ZZ[X_1,\dots,X_n]$, we set 
\begin{equation}\label{eq:E(gH)}
E(g;H)=\card \{\x\in Y(\ZZ): |\x|\leq H, ~ g(\x)=0\}.
\end{equation}
Breaking  the first cardinality into congruence classes modulo $k^r$, we conclude that
$$
N^{(1)}(H)
=\sum_{k\leq H^\Delta} \mu(k) \sum_{\substack{\bxi\in Y(\ZZ/k^r\ZZ)\\
f(\bxi)\equiv 0\bmod{k^r}}}
V_{k^r}(H;\bxi) +O\left(H^{\Delta} E(f;H)\right),
$$
where 
for any  $\ell\in \NN$ and  $\bxi\in Y(\ZZ/\ell\ZZ)$, we put
$$
V_\ell(H;\bxi)=\card \{\x\in Y(\ZZ): |\x|\leq H, ~ \x\equiv \bxi \bmod{\ell}
\}.
$$
We now shift our attention to estimating $V_\ell(H;\bxi)$, as $H\to\infty$.

\begin{prop}\label{p:main'}
There exists $\delta>0$ such that 
$$
V_\ell(H;\bxi)=\mu_\infty(Y;H) \prod_{p<\infty} \hat \mu_p(Y;\bxi;\ell)  +O(\ell^{\dim(L)+2\dim(G)} \mu_\infty(Y;H)^{1-\delta}),
$$
where $\mu_\infty(Y;H)$ is defined in  \eqref{eq:real} and 
$$
\hat \mu_p(Y;\bxi,\ell)=\lim_{t\rightarrow \infty} p^{-t\dim(Y)}\#\{\x\in Y(\ZZ/p^t\ZZ):  \x\equiv \bxi \bmod{p^{v_p(\ell)}}\}.
$$
The implied constant in this estimate depends only on $Y$ and is independent of $\ell$ and $\bxi$.
\end{prop}

\begin{proof}
	This result is deduced from our work in  \S\ref{s:counting}.
	While there we stated the estimates in terms of the measures $m_Y$ and $m_{Y,f}=\prod_{p<\infty} m_{Y,p}$, but they can also be interpreted using local densities. 
	By \cite[Lemma~1.8.2]{borovoi},
	$$
	m_Y(Y(\RR)_H)=\mu_\infty(Y;H).
	$$
	Also, the proof of \cite[Lemma~1.8.2]{borovoi} gives 
	\begin{equation}
	\label{eq:U}
	m_{Y,p}(U)=\lim_{t\rightarrow \infty} p^{-t\dim(Y)}\#(U\;\hbox{mod}\; p^t),
	\end{equation}
	for every open $U\subset Y(\ZZ_p)$.
	In particular,
	$$
	m_{Y,p}(B_p(\bxi,\ell))=\hat \mu_p(Y;\bxi,\ell).
	$$
	The result now 
	follows from Corollary \ref{cor:variety}.
\end{proof}

Next,  we claim that 
\begin{equation}\label{eq:deck}
\prod_{p<\infty} \hat \mu_p(Y;\bxi;\ell) \ll \frac{1}{\ell^{\dim(Y)}},
\end{equation}
for any $\ell\in \NN$. 
Let $p\mid \ell$ and let $\mu=v_p(\ell)$. 
Recall that $Y$ is non-singular and let $p$ be a prime of good reduction for $Y$.
We set
$$
N(p^t)=\#\{\x\in Y(\ZZ/p^t\ZZ):  \x\equiv \bxi \bmod{p^{v_p(\ell)}}\}.
$$
It follows from Hensel's lemma that $N(p^{t+1})=p^{\dim(Y)}N(p^t)$ for any $t\geq \mu$. Hence
$\hat \mu_p(Y;\bxi;\ell)=p^{-\mu\dim(Y)}N(p^\mu)=p^{-\mu\dim (Y)}$, since $\bxi\in Y(\ZZ/p^\mu\ZZ)$.
The claim now easily follows.

Now there are  $\rho(k^r)\leq k^{rn}$ choices of $\bxi$  which we must consider. 
We substitute the estimate from Proposition~\ref{p:main'} to obtain
\begin{align*}
N^{(1)}(H)
=~&\mu_\infty(Y;H) S(H)\\ &+
O\left(
H^{\Delta(1+r\{n+\dim(L)+2\dim(G)\})} \mu_\infty(Y;H)^{1-\delta}+H^\Delta E(f;H)\right),
\end{align*}
where 
$E(f;H)$ is given by \eqref{eq:E(gH)}
and
\begin{equation}\label{eq:sugar}
S(H)
=\sum_{k\leq H^\Delta} \mu(k) 
\sum_{\substack{\bxi\in Y(\ZZ/k^r\ZZ)\\f(\bxi)\equiv 0\bmod{k^r}}}
\prod_{p<\infty} \hat \mu_p(Y;\bxi;k^r).
\end{equation}

\begin{lemma} \label{l:e}
Let $g\in \ZZ[X_1,\ldots,X_n]$ be  such that $g\not\equiv  0$ on $Y$.
Then there exists $\eta>0$ such that  
$
E(g;H)=O_g(\mu_\infty(Y;H)^{1-\eta}),
$
where $E(g;H)$ is given by \eqref{eq:E(gH)}.
\end{lemma}

\begin{proof}
	
Pick a large prime $p$. It follows from Proposition \ref{p:main'} and \eqref{eq:deck} that 
\begin{align*}
E(g;H)&\leq \card \{\x\in Y(\ZZ): |\x|\leq H, ~ p\mid g(\x)\}
\\
&=
\sum_{\substack{\bxi\in Y(\ZZ/p\ZZ)\\g(\bxi)\equiv 0\bmod{p}}} V_p(H;\bxi)\\
&\ll 
\sum_{\substack{\bxi\in Y(\ZZ/p\ZZ)\\g(\bxi)\equiv 0\bmod{p}}}\left(
 \frac{\mu_\infty(Y;H)}{p^{\dim(Y)}}+
p^{\dim(L)+2\dim(G)} \mu_\infty(Y;H)^{1-\delta}
\right),
\end{align*}
for some $\delta>0$. 
Since $Y$ is irreducible, $\dim (Y\cap \{g=0\})<\dim(Y)$ and
 so it follows from the Lang-Weil estimates that
$$
\card \{\bxi\in Y(\ZZ/p\ZZ):\;g(\bxi)\equiv 0\bmod{p}\}\ll_g p^{\dim(Y)-1}.$$
Hence,
\begin{align*}
E(g;H)
&\ll_g 
 \frac{\mu_\infty(Y;H)}{p}+
p^{3\dim(G)-1} \mu_\infty(Y;H)^{1-\delta},
\end{align*}
since $\dim(Y)=\dim(G)-\dim(L)$. 
This is satisfactory for the lemma, on choosing $p$ appropriately. 
\end{proof}

We may now conclude that there exists $\delta>0$ such that 
\begin{align*}
N^{(1)}(H)
=\mu_\infty(Y;H) S(H)+
O\left(
H^{\Delta(1+r\{n+\dim(L)+2\dim(G)\})} \mu_\infty(Y;H)^{1-\delta}\right),
\end{align*}
where $S(H)$ is given by \eqref{eq:sugar}.
Turning to an analysis of $S(H)$, we appeal to 
 \Hyp, which gives
$$
\rho(k^r)\leq C_{Y,f,r}^{\omega(k)}k^{r(\dim(Y)-1)}\ll_{\ve,r} k^{r(\dim(Y)-1)+\ve}, 
$$
for an appropriate constant $C_{Y,f,r}>0$.
Invoking \eqref{eq:deck} with $\ell=k^r$, and recalling that $r\geq 2$,  we may 
therefore extend the sum over $k$ to infinity, finding that 
$$
S(H)
=\sum_{k=1}^\infty \mu(k) 
\sum_{\substack{\bxi\in Y(\ZZ/k^r\ZZ)\\f(\bxi)\equiv 0\bmod{k^r}}}
\prod_{p<\infty} \hat \mu_p(Y;\bxi;k^r) +O(H^{-\Delta/2}).
$$
The  main term here is equal to the Euler product $\mathfrak{S}(Y,f,r)$  that is defined in \eqref{eq:euler0}. 
Putting everything together, we have therefore established the following result, which completes our treatment of the small moduli.

\begin{prop}\label{lem:dynamics}
Let $Y\simeq G/L\subset \AA^n$ be a symmetric variety over $\QQ$ satisfying  (i)--(iv), 
with $Y(\ZZ)\neq \emptyset$, and with $L$  semisimple and simply connected.
Assume that $f\in \ZZ[X_1,\dots,X_n]$ satisfies Hypothesis-$\rho$. Then 
there exists $\delta>0$  such that
\begin{align*}
N^{(1)}(H)=&
\mathfrak{S}(Y,f,r) \mu_{\infty}(Y;H)\\
&+
O\left(
H^{\Delta(1+r\{n+\dim(L)+2\dim(G)\})}\mu_{\infty}(Y;H)^{1-\delta}  
+\mu_{\infty}(Y;H) H^{-\Delta/2}
\right).
\end{align*}
Moreover, $\mathfrak{S}(Y,f,r)>0$ provided that $f$ has no $r$-power divisors on $Y$.
\end{prop}

\subsection{Proof of Theorem \ref{t:symmetric2}}

Assuming $\Delta>0$ is chosen to be sufficiently small in terms of $\delta, \dim(G)$ and $r$, the error terms in Proposition \ref{lem:dynamics} can both be made  smaller than the main term.
It remains to show that the 
contribution
\begin{equation}\label{eq:leek2}
N^{(2)}(H)=
\hspace{-0.5cm}
\sum_{H^\Delta<k\ll H^{d/r}} 
\hspace{-0.5cm}
|\mu(k)| 
\card \{\x\in Y(\ZZ): |\x|\leq H, ~ 0\neq f(\x)\equiv 0\bmod{k^r}\}
\end{equation}
is negligible.   
Here, we have truncated the outer sum to $k\ll H^{d/r}$, 
on supposing that  $f$ has degree $d$. 
As remarked in \S \ref{s:intro}, our key observation for handling large moduli is based on the inequality
\begin{align*}
N^{(2)}(H)
&\leq \sum_{H^\Delta<k\ll H^{d/r}} |\mu(k)| 
\card \{\x\in Y(\ZZ): |\x|\leq H, ~ f(\x)\equiv 0\bmod{k^2}\}\\
&= \sum_{H^\Delta<k\ll H^{d/r}} |\mu(k)| 
\sum_{\substack{\bxi\in Y(\ZZ/k^2\ZZ)\\f(\bxi)\equiv 0\bmod{k^2}}}
V_{k^2}(H;\bxi),
\end{align*}
in the notation of Proposition \ref{p:main'}. Combining this result with \eqref{eq:deck} and \Hyp, we therefore conclude that 
\begin{align*}
N^{(2)}(H)
&\ll 
\hspace{-0.2cm}
\sum_{H^\Delta<k\ll H^{d/r}} 
\hspace{-0.2cm}
|\mu(k)| \rho(k^2) \left\{
\frac{\mu_\infty(Y;H)}{k^{2\dim(Y)}}
+ k^{2\dim(L)+4\dim(G)}  \mu_\infty(Y;H)^{1-\delta}\right\}\\
&\ll 
\hspace{-0.2cm}
\sum_{H^\Delta<k\ll H^{d/r}} 
\hspace{-0.2cm}\left\{
\frac{\mu_\infty(Y;H)}{k^2}
+ k^{6\dim(G)}  \mu_\infty(Y;H)^{1-\delta}\right\}\\
&\ll 
\frac{\mu_\infty(Y;H)}{H^\Delta}
+ H^{\frac{d}{r}(1+6\dim(G))}  \mu_\infty(Y;H)^{1-\delta}.
\end{align*}
The first term is satisfactory and the second term is also satisfactory provided that $r$ is taken to be sufficiently large in terms of $d, \dim(G)$ and $\delta$. This completes the proof of Theorem \ref{t:symmetric2}.

\subsection{Generalisation of Theorem \ref{t:symmetric2} and proof of Theorem \ref{t:symmetric}}
\label{s:gene}
In this section we discuss symmetric varieties $Y\simeq G/L$ when $L$ is not necessarily a semisimple simply connected group. This requires a more delicate analysis because such varieties may fail to satisfy 
the Hardy--Littlewood asymptotic formula. Throughout this section we assume that conditions (i)--(iv) from \S  \ref{s:intro} hold.
Let $\mathcal{O}_{\A}=\prod'_{p\le \infty} \mathcal{O}_p\subset Y(\A)$ be an orbit of $G(\A)$ in $Y(\A)$.
Our goal is to estimate the counting function
$$
N_r(\mathcal{O}_{\A};f;H)=\card\{\x\in Y(\ZZ)\cap \mathcal{O}_{\A} : |\x|\leq H, ~\text{$f(\x)$ is $r$-free}\}.
$$
We introduce local densities associated to the orbit $\mathcal{O}_{\A}$:
$$
\hat \mu_p(\mathcal{O}_\A,f,r)=\lim_{t\rightarrow \infty} p^{-t\dim(Y)}\#\{\x\in Y(\ZZ_p)\cap \mathcal{O}_\A \; \hbox{mod}\; p^t:  ~p^r\nmid f(\x)\}.
$$
For almost all $p$, $\mathcal{O}_p\supset Y(\ZZ_p)$ which implies that $\hat \mu_p(\mathcal{O}_\A,f,r)=\hat \mu_p(Y,f,r)$.
We also define the corresponding Euler product 
$$
\mathfrak{S}(\mathcal{O}_{\A},f,r)=L(1,\rho_L)\prod_{p<\infty} L_p(1,\rho_L)^{-1} \hat \mu_p(\mathcal{O}_\A,f,r),
$$
where $L(s,\rho_L)$ is given by \eqref{eq:artin}.
Letting 
$\hat \mu_p(Y)=m_{Y,p}(Y(\ZZ_p))$, in the notation of \eqref{eq:U}, 
it follows from Hypothesis-$\rho$ that the product
$$
\prod_{p<\infty} \frac{\hat \mu_p(Y,f,r)}{\hat \mu_p(Y)}
$$
converges absolutely,
so that the Euler product $\mathfrak{S}(\mathcal{O}_{\A},f,r)$ also converges absolutely.

With this notation, we establish the following result.

\begin{theorem}\label{th:general}
Let $Y\simeq G/L\subset \AA^n$ be a symmetric variety over $\QQ$ satisfying  (i)--(iv).
Let $\mathcal{O}_{\A}\subset Y(\A)$ be an orbit of $G(\A)$.
Assume that the polynomial $f\in \ZZ[X_1,\dots,X_n]$ satisfies Hypothesis-$\rho$.
Then for all sufficiently large $r$, there exists $\delta>0$  such that
\begin{align*}
N_r(\mathcal{O}_{\A};f;H)=
\delta(\mathcal{O}_{\A})\mathfrak{S}(\mathcal{O}_{\A},f,r) \mu_Y(\mathcal{O}_\infty(H))
+
O_r\left(
\mu_Y(\mathcal{O}_\infty(H))^{1-\delta}  
\right).
\end{align*}
Moreover, if we assume that there exists $\x\in Y(\ZZ)\cap \mathcal{O}_{\A}$ such that $f(\x)$ is $r$-free, then 
$\delta(\mathcal{O}_{\A})>0$ and $\mathfrak{S}(\mathcal{O}_{\A},f,r)>0$.
\end{theorem}

If the variety $Y$ additionally satisfies the smoothness assumption \eqref{eq:smooth},
then it follow from the argument in  \cite[Lemma~1.8.2]{borovoi} that 
$$
\mu_Y(\mathcal{O}_\infty(H))=\mu_\infty(\mathcal{O}_{\A};H),
$$
where the local density $\mu_\infty(\mathcal{O}_{\A};H)$ is defined analogously to \eqref{eq:real2}.

\begin{rem}\label{r:BM}
We note that 
$$
N_r(Y;f;H)=\sum_{\mathcal{O}_{\A}\subset Y(\A)} N_r(\mathcal{O}_{\A};f;H),
$$
where the sum is taken over finitely many orbits $\mathcal{O}_{\A}$
that have non-trivial intersection with $Y(\RR)\times \prod_{p<\infty} Y(\ZZ_p)$.
Hence, Theorem \ref{th:general} implies an asymptotic formula for the counting function
$N_r(Y;f;H)$. In fact, this asymptotic formula can be stated in terms of the Tamagawa volume of a suitable subset of $Y(\A)$ defined by the integral Brauer--Manin obstruction, as introduced by Colliot-Th\'el`ene and Xu \cite{cx}.
We denote by $Y(\A)^{{\rm Br}(Y)}$ the kernel of the Brauer--Manin pairing. 
Since $G$ is assumed to be simply connected, this kernel consists of orbits
of $G(\A)$ (see \cite[Thm.~3.2]{cx}). If $\mathcal{O}_\A \cap Y(\A)^{{\rm Br}(Y)}=\emptyset$,
then the orbit $\mathcal{O}_\A$ contains no rational points, and 
$\delta(\mathcal{O}_{\A})=0$. On the other hand, if $\mathcal{O}_\A \subset Y(\A)^{{\rm Br}(Y)}$, and  $\mathcal{O}_{\A}$
has non-trivial intersection with $Y(\RR)\times \prod_{p<\infty} Y(\ZZ_p)$,
then it follows from \cite[Thm.~3.7]{cx} that $\mathcal{O}_{\A}\cap Y(\ZZ)\ne \emptyset$.
In particular, we conclude that for these orbits $\delta(\mathcal{O}_{\A})>0$.
Since $\delta:Y(\A)\to \{0,|\hbox{Pic}(L)|\}$, we have $\delta(\mathcal{O}_{\A})=|\hbox{Pic}(L)|$ for these orbits.
Thus, setting 
$$
Y(\A_f)^{(f,r)}=\left\{(\y_p)\in \prod_{p<\infty} Y(\ZZ_p):\; f(\y_p)\not\equiv 0\; (\hbox{mod}\, p^r) \right\},
$$
we conclude that 
\begin{align*}
N_r(Y;f;H)=&|\hbox{Pic}(L)| \cdot (m_Y\otimes m_{Y,f}) \left((Y(\RR)_H\times Y(\A_f)^{(f,r)}) \cap Y(\A)^{{\rm Br}(Y)}\right)\\
&+O(V(H)^{1-\delta}),
\end{align*}
where
$
V(H)=\max_{\mathcal{O}_\A\subset Y(\A)} \mu_{\infty}(\mathcal{O}_{\A};H),
$
with the maximum  taken over the finitely many orbits
 having non-trivial intersection with $Y(\RR)\times \prod_{p<\infty} Y(\ZZ_p)$.
\end{rem}

\begin{proof}[Proof of Theorem \ref{th:general}]
The proof follows the same strategy as the proof of Theorem \ref{t:symmetric2} presented in \S3.1--3.2.
As in the proof of Theorem \ref{t:symmetric2}, we start by estimating
\begin{align*}
N^{(1)}(\mathcal{O}_{\A};H)
=~&
\sum_{k\leq H^\Delta} 
\mu(k) 
\card \{\x\in Y(\ZZ)\cap \mathcal{O}_{\A}: |\x|\leq H, ~ f(\x)\equiv 0\bmod{k^r}
\}\\
&+ O\left(H^{\Delta}E_{\mathcal{O}_{\A}}(f;H)\right),
\end{align*}
where
$
E_{\mathcal{O}_{\A}}(f;H)=\card \{\x\in Y(\ZZ)\cap \mathcal{O}_{\A}: |\x|\leq H, ~ f(\x)=0\}.
$
Then
$$
N^{(1)}(\mathcal{O}_{\A};H)
=\sum_{k\leq H^\Delta} \mu(k) \sum_{\substack{\bxi\in Y(\ZZ/k^r\ZZ)\\
		f(\bxi)\equiv 0\bmod{k^r}}}
V_{k^r}(\mathcal{O}_{\A};H;\bxi) + O\left(H^{\Delta}E_{\mathcal{O}_{\A}}(f;H)\right), 
$$
where
$$
V_\ell(\mathcal{O}_{\A};H;\bxi)=\card \{\x\in Y(\ZZ)\cap \mathcal{O}_{\A}: |\x|\leq H, ~ \x\equiv \bxi \bmod{\ell}
\},
$$
for given $\ell$ and  $\bxi\in Y(\ZZ/\ell\ZZ)$. 
The quantity $V_\ell(\mathcal{O}_{\A};H;\bxi)$ can be estimated as in Proposition \ref{p:main'}.
We obtain that there exists $\delta>0$ such that 
\begin{equation}\begin{split}\label{eq:vl}
	V_\ell(\mathcal{O}_{\A};H;\bxi)=~&\delta(\mathcal{O}_{\A})\mu_Y(\mathcal{O}_\infty(H)) \mathfrak{S}(\mathcal{O}_{\A},\bxi;\ell)  \\
	&+O(\ell^{\dim(L)+2\dim(G)} \mu_Y(\mathcal{O}_\infty(H))^{1-\delta}),
\end{split}
\end{equation}
where 
$$
\mathfrak{S}(\mathcal{O}_{\A},\bxi;\ell)=L(1,\rho_L)\prod_{p<\infty} L_p(1,\rho_L)^{-1} \hat \mu_p(\mathcal{O}_\A,\bxi;\ell)
$$
and
$$
\hat \mu_p(\mathcal{O}_{\A};\bxi,\ell)=\lim_{t\rightarrow \infty} p^{-t\dim (Y)}\#\{\x\in \mathcal{O}_p\cap Y(\ZZ_p)\;\hbox{\rm mod}\; p^t:  \x\equiv \bxi \bmod{p^{v_p(\ell)}}\}.
$$
Indeed, this estimate can be directly deduced from Corollary \ref{cor:orbit}
by observing that
$$
	m_{Y,p}(\mathcal{O}_p\cap B_p(\bxi,\ell))=\hat \mu_p(\mathcal{O}_{\A};\bxi,\ell),
$$
which follows from \eqref{eq:U}. Next, we substitute \eqref{eq:vl} into our work above to deduce that
\begin{align*}
N^{(1)}(\mathcal{O}_{\A};H)
=~&\delta(\mathcal{O}_{\A})\mu_Y(\mathcal{O}_\infty(H)) S(\mathcal{O}_{\A};H)
+O\left(H^\Delta
E_{\mathcal{O}_{\A}}(f;H)\right)\\
&+
O\left(
H^{\Delta(1+r\{n+\dim(L)+2\dim(G)\})} \mu_Y(\mathcal{O}_\infty(H))^{1-\delta}\right),
\end{align*}
where
$$
S(\mathcal{O}_{\A};H)
=\sum_{k\leq H^\Delta} \mu(k) 
\sum_{\substack{\bxi\in Y(\ZZ/k^r\ZZ)\\f(\bxi)\equiv 0\bmod{k^r}}}
\mathfrak{S}(\mathcal{O}_{\A},\bxi;k^r).
$$
Arguing as in the proof of Theorem \ref{t:symmetric2}, we find that 
$$
S(\mathcal{O}_{\A};H)
=\sum_{k=1}^\infty \mu(k) 
\sum_{\substack{\bxi\in Y(\ZZ/k^r\ZZ)\\f(\bxi)\equiv 0\bmod{k^r}}}
\mathfrak{S}(\mathcal{O}_{\A},\bxi;k^r) +O(H^{-\Delta/2}),
$$
with  the main term is equal to the Euler product $\mathfrak{S}(\mathcal{O}_{\A},f,r)$.

Arguing as in Lemma \ref{l:e}, we easily use \eqref{eq:vl} to show that 
\begin{equation}
\label{eq:g}
E_{\mathcal{O}_{\A}}(f;H)=O_g(\mu_Y(\mathcal{O}_\infty(H))^{1-\eta}),
\end{equation}
for some $\eta>0$.
Hence, we conclude that 
\begin{equation}\begin{split} \label{eq:final0}
N^{(1)}(\mathcal{O}_{\A};H)=~&
\delta(\mathcal{O}_{\A})\mathfrak{S}(\mathcal{O}_{\A},f,r) \mu_Y(\mathcal{O}_\infty(H))
+O\left(\mu_Y(\mathcal{O}_\infty(H)) H^{-\Delta/2}\right)
\\
&+
O\left(
H^{\Delta(1+r\{n+\dim(L)+2\dim(G)\})}\mu_Y(\mathcal{O}_\infty(H))^{1-\delta}  
\right),
\end{split}\end{equation}
for every adelic orbit $\mathcal{O}_{\A}\subset Y(\A)$.
When $\Delta>0$ is chosen sufficiently small, the error terms in this estimate
can be made smaller than the main term. 

Now it remains to estimate 
\begin{align*}
N^{(2)}(\mathcal{O}_{\A};H)
&=
\hspace{-0.5cm}
\sum_{H^\Delta<k\ll H^{d/r}} 
\hspace{-0.5cm}
|\mu(k)| 
\card \{\x\in Y(\ZZ)\cap \mathcal{O}_{\A} : |\x|\leq H, ~  f(\x)\equiv 0\bmod{k^r}\}\\
&\le \sum_{H^\Delta<k\ll H^{d/r}} |\mu(k)| 
\sum_{\substack{\bxi\in Y(\ZZ/k^2\ZZ)\\f(\bxi)\equiv 0\bmod{k^2}}}
V_{k^2}(\mathcal{O}_{\A};H;\bxi).
\end{align*}
Thus we can argue as in \S3.2, using \eqref{eq:vl}, combined with Hypotheses-$\rho$ and \eqref{eq:deck}, to conclude that if $r$ is taken sufficiently large, there exists $\delta'>0$ such that
$$
N^{(2)}(\mathcal{O}_{\A};H)=O(\mu_Y(\mathcal{O}_\infty(H))^{1-\delta'}).
$$
This completes the proof of the theorem.
\end{proof}

\begin{proof}[Proof of Theorem \ref{t:symmetric}]
Since $Y(\ZZ)\ne \emptyset$, we may pick $\x\in Y(\ZZ)$ with $f(\x)\neq 0$. Suppose that $|f(\x)|=\prod_{1\leq i\leq \ell}p_i^{r_i}$, for distinct primes $p_1,\dots,p_\ell$. Then $f(\x)$ is $r_0$-free, with $r_0=1+\max_{1\leq i\leq \ell}\{r_i\}$. 
We apply Theorem \ref{th:general} to an orbit $\mathcal{O}_\AA$ that contains this point $\x$. Then 
$\delta(\mathcal{O}_{\A})>0$ and $\mathfrak{S}(\mathcal{O}_{\A},f,r)>0$ if $r\geq r_0.$
Comparing the asymptotic formula given by Theorem \ref{th:general} with
\eqref{eq:g} we deduce that the set 
$\{\x\in Y(\ZZ)\cap \mathcal{O}_{\A} : \text{$f(\x)$ is $r$-free}\}$
is Zariski dense in $Y$.
(We note that $\mu_Y(\mathcal{O}_\infty(H))\to \infty$, as $H\to \infty$, by Lemma \ref{l:inf}.)
\end{proof}

\section{The roadmap for quadrics}\label{s:over}

It is now time to initiate the proof of Theorems \ref{thm:2large}, \ref{thm:2} and \ref{thm:1}.  From this point forwards,  $n\geq 3$ and $Y\subset \AA^n$ is the affine quadric \eqref{eq:green}, where
 $Q\in \ZZ[X_1,\ldots,X_n]$
is a non-singular indefinite quadratic form and 
$m$ is  a non-zero integer such that 
 $-m\det (Q)\neq \square$ when $n=3$.
 We begin with a proof of 
Theorem \ref{thm:3}  in \S \ref{s:dim}.  Next, in \S \ref{s:poly}, we shall establish \Hyp~for the polynomials $f\in \ZZ[X_1,\dots,X_n]$ under consideration. Finally, in \S \ref{s:onion} we shall collect together the main steps in the proof of 
Theorems \ref{thm:2large}, \ref{thm:2} and \ref{thm:1}. The primary ingredients in this endeavour here are the results in \S \ref{s:small} and the treatment of large moduli in \S \ref{s:largesse}.

\subsection{Integral points on affine  quadrics}\label{s:dim} 

In this section we establish Theorem \ref{thm:3}. 
We begin by noting that it suffices to assume that $q$ is absolutely irreducible in the statement of the theorem, rather than merely irreducible over $\QQ$. Indeed, if $q$ factorises as $\ell_1\ell_2$ for linear polynomials $\ell_1,\ell_2 \in \overline{\QQ}[T_1,\ldots,T_\nu]$,   neither one of which is proportional to a linear polynomial defined over $\ZZ$, then $\ell_1,\ell_2$ are not proportional to each other and 
$\ell_1$ must be a conjugate of  $\ell_2$. Moreover,  the integer points $\bt$ in which we are interested must satisfy the pair of equations $\ell_1(\bt)=\ell_2(\bt)=0$. Such points clearly 
contribute only $O_\nu(B^{\nu-2})$ to $M(q;B)$, which is satisfactory.

With the restriction to absolutely irreducible $q$ in place, we will establish Theorem \ref{thm:3} by induction on $\nu\geq 2$, following the approach in \cite[\S 4]{pila}.
We henceforth set
\begin{equation}\label{eq:R}
R(X_0,X_1,\ldots,X_\nu)=X_0^2q(X_1/X_0,\ldots,X_{\nu}/X_0)
\end{equation}
for the homogenised quadratic form associated to $q$. 
In particular $q_0$ is obtained by setting $X_0=0$ in $R$.
Since $q$ is absolutely irreducible it follows that $R$ is absolutely irreducible and so has rank at least $3$. Moreover, by hypothesis,  the quadratic form
$R(0,X_1,\ldots,X_\nu)$ has rank at least $2$.  We will need the following result,
due to Browning, Heath-Brown and Salberger
 \cite[Lemma 13]{pila}.

\begin{lemma}\label{lem:affquad}
Let $\ve>0$, let $B\geq 1$ and suppose that
$R\in \ZZ[X_0,X_1,X_2]$
is a non-singular quadratic form such that 
the binary form 
$R(0,X_{1},X_{2})$ is also
non-singular.  
Then for any  $t\in \ZZ\cap[-B,B]$ we have
$$
\card\left\{
(x_1,x_2) \in \ZZ^2: 
\begin{array}{l}
|x_1|,|x_2| \leq B, ~R(t,x_1,x_2)=0\\
\gcd(t,x_1,x_2)=1
\end{array}
\right\}
=O_\ve(B^{\ve}).
$$
\end{lemma}

Following our convention, the  implied constant in this estimate does not depend 
on $t$ or on the coefficients of $R$.
The case $\nu=2$ of Theorem~\ref{thm:3} is now a trivial consequence of Lemma \ref{lem:affquad} with $t=1$.
We will require  a separate treatment of the case $\nu=3$ when $q$ is absolutely irreducible with $q=q_0$.
In this case the statement of Theorem \ref{thm:3} follows from taking $d=2$ in work of Heath-Brown \cite[Thm.~3]{annals}.

We now turn to the proof of Theorem \ref{thm:3} when $\nu \geq 3$, assuming that  $q\neq q_0$ when $\nu=3$.
Our plan is to take hyperplane slices and apply the inductive hypothesis. 
We claim that there 
exists $\a\in \ZZ^\nu$, with $0<|\a|=O_\nu(1)$, such that the quadratic polynomial obtained by eliminating a variable from the pair of equations
\begin{equation}\label{eq:pair}
q(T_1,\ldots,T_\nu)=a_1T_1+\cdots+a_\nu T_\nu=0
\end{equation}
is  absolutely irreducible and has quadratic part with rank at least $2$. 
Taking this claim on faith for the moment, we may assume after a possible change of variables that $\a=(0,\ldots,0,1)$.  Thus, for any $k\in \ZZ$, the polynomial  $q_k=q(T_1,\ldots,T_{\nu-1},k)$ is both absolutely irreducible and has quadratic part with  rank at least $2$.
In this way we obtain the bound
$$
M(q;B)\leq \sum_{|k|\leq B} M(q_k;B) \ll_{\ve,\nu} B^{\nu-2+\ve}, 
$$
by the inductive hypothesis. This completes the proof of the theorem subject to the claim.

Let us call a vector $\a\in \CC^\nu$ {\em defective}
if the polynomial induced by \eqref{eq:pair} fails to be absolutely irreducible or has quadratic part with rank at most $1$.  We will construct a proper subvariety $E\subset \PP^{\nu-1}$ defined over $\ZZ$, 
with degree $O_\nu(1)$, such that $[\a]\in E$ whenever $\a$ is defective. 
Once this is achieved it is a simple matter to find  a vector $\a\in \ZZ^\nu$ satisfying the claim. 
Indeed, 
for any $A>1$, there are at least $c_1(\nu ) A^\nu$ possible non-zero 
vectors $\a\in \ZZ^\nu$ for which $|\a|\leq A$, for an appropriate constant $c_1(\nu)>0$. Moreover, it follows from the trivial estimate
\cite[Lemma 2]{pila} that there are 
at most $c_2(\nu)A^{\nu-1}$  defective vectors $\a\in \ZZ^\nu$ 
satisfying $|\a|\leq A$, for an appropriate constant $c_2(\nu)>0$.
The claim then follows on 
taking $A>c_2(\nu)/c_1(\nu)$.

It remains to  construct the variety $E$.   
Let us begin by considering vectors $\a\in \CC^\nu$ for which \eqref{eq:pair} is not absolutely irreducible. 
When $q$ is homogeneous, so that $q=q_0$ and $\nu\geq 4$, then  it is 
well-known (see \cite[Lemma 7]{pila}, for example) 
that there exists a non-zero form 
$F\in \ZZ[X_1,\ldots,X_\nu]$, with degree $O_\nu(1)$, 
such that $F(\a)= 0$ when the intersection is not absolutely irreducible. 
Alternatively, when $q\neq q_0$ and $\nu \geq 3$, we will  work with the homogenised quadratic form \eqref{eq:R}.
Let $U\subset \PP^\nu$ denote the quadric $R=0$.
Using elimination theory we can construct a form 
$F\in \ZZ[X_1,\ldots,X_\nu]$, with degree $O_\nu(1)$, 
such that $F(\a)= 0$ whenever the intersection of $U$ with the 
hyperplane $\sum_{i=1}^\nu a_iX_i=0$ produces a reducible quadric.  
We need to show  that $F$ is non-zero. 
Let $x=[1,0,\ldots,0]$ and let
$\Sigma_x$ denote the set of hyperplanes in $\PP^\nu$ containing $x$.
Then the desired conclusion follows from the version of Bertini's theorem found in 
Fulton and Lazarsfeld \cite[Thm.~1.1]{F-L}, which shows 
that $U\cap H$ is absolutely irreducible for generic $H\in \Sigma_x$.
We let $E_1\subset \PP^{\nu-1}$  denote the  projective hypersurface $F=0$.

Shifting attention to
the vectors $\a\in \CC^\nu$ for which \eqref{eq:pair} has quadratic part with rank at most $1$,
 the $2\times 2$  minors of the underlying quadratic form give a  system of six
 homogeneous quadratic equations whose simultaneous vanishing at $\a$ encapsulates this property. We denote this variety
 by $E_2\subset \PP^{\nu-1}$.  After verifying that $E_2$ is a proper subvariety, 
our construction is completed by taking $E=E_1\cup E_2$.

Let $V\subset\PP^{\nu-1}$ denote the quadric $q_0=0$. Then $\nu \geq 3$ and 
$V$ has rank $r_V\geq 2$, on
identifying the rank of a quadric with the rank of the underlying quadratic form.   
To prove that $E_2\neq \PP^{\nu-1}$ it suffices to show that for generic hyperplanes $H$ in $ \PP^{\nu-1}$ the intersection $V\cap H$ produces a quadric with rank at least $2$. 
Thus we need to know some elementary facts about 
how the ranks of quadratic forms  diminish on linear subspaces.
For a hyperplane $H$ let $W=V\cap H$ and let $r_W$ be the associated rank. 
It is well-known that $r_W\geq r_V-2$.  
We will need slightly finer information (see Swinnerton-Dyer \cite[p.~264]{swd}, for example).
If $r_V=\nu$  then 
$r_W=
\nu-1$ if $H$ is not tangent to $V$. 
If $r_V<\nu$
then the singular points of $V$ form a linear space $L$ of dimension $\nu-r_V-1$ and we have 
$r_W=
r_V$ if $L\not\subset H$.
In either case we  deduce that $r_W\geq 2$ for generic $H$.
This concludes the proof of Theorem~\ref{thm:3}.

\subsection{Polynomial congruences modulo prime powers}\label{s:poly}

This section is concerned with counting solutions to certain systems of polynomial congruences modulo prime powers. 
Recalling the definition \eqref{eq:def-rho} of $\rho(\ell)$, 
we begin by establishing \Hyp~for the non-singular forms  $f\in \ZZ[X_1,\dots,X_n]$ that feature in Theorems \ref{thm:2large}, \ref{thm:2} and \ref{thm:1}, when $Y\subset \AA^n$ is the affine quadric \eqref{eq:green}. (Note that a linear form is automatically non-singular.)

\begin{lemma}\label{lem:rho'}
Let  $f\in \ZZ[X_1,\dots,X_n]$ be a  non-singular form of degree $d\geq 1$.
Then we have
$
\rho(p^r)\ll_r p^{r(n-2)}.
$
\end{lemma}

\begin{proof}
The trivial bound is 
$
\rho(p^r)\leq p^{rn}.
$
Since we allow our implied constant 
 to depend on $r$, we may henceforth assume that 
$
p\nmid 2dm\Delta_f\det(Q),
$ 
where $\Delta_f$ is the discriminant of $f$.

When $r=1$ 
it follows from the Lang--Weil estimate that $\rho(p)=O(p^{n-2})$. When $r\geq 1$ 
 the statement of the lemma will follow provided we can show that 
$$
\rho(p^{r+1})=p^{n-2}\rho(p^r).
$$
To verify this we use an approach based on Hensel's lemma. 
Let $\x \bmod{p^{r}}$ be counted by $\rho(p^r)$ and consider the vectors 
$\x+p^r\y$ for $\y\bmod{p}$. Such a vector 
runs modulo $p^{r+1}$ and is counted by $\rho(p^{r+1})$ if and only if 
\begin{align*}
p^{-r}f(\x)+\y.\nabla f(\x)&\equiv 0\bmod{p}, \\ 
p^{-r}(Q(\x)-m)+\y.\nabla Q(\x)&\equiv 0 \bmod{p}.
\end{align*}
We claim that 
$\nabla f(\x)$ and $\nabla Q(\x)$ are not proportional modulo $p$, from which it will follow
that  there are  $p^{n-2}$
possibilities for $\y \bmod{p}$, as required. Suppose for a contradiction  that
there exists $\lambda,\mu\in \FF_p$, not both zero,  such that 
$\lambda \nabla f(\x)\equiv \mu \nabla Q(\x)\bmod{p}$. Since $p\nmid m$  we must have $p\nmid \x$. In particular, $\lambda\mu\neq 0$ since $f$ and $Q$ are non-singular modulo $p$. 
It then follows from Euler's identity that 
\begin{align*}
0&\equiv d \lambda f(\x) \bmod{p} \\
&\equiv \lambda\x.\nabla f (\x) \bmod{p}\\
&\equiv  \mu \x.\nabla Q(\x)  \bmod{p}\\
&\equiv 2 \mu m  \bmod{p},
\end{align*}
which is a contradiction. 
\end{proof}


When $n=3$ it turns out that we shall also need a good bound for
$$\rho(\ell;\c)=
\#\left\{
\x\in Y(\ZZ/\ell\ZZ): 
f(\x)\equiv 0 \bmod{\ell},~
 \c.\x\equiv 0\bmod{\ell}
\right\},
$$
for any $\ell\in \NN$ and $\c\in \ZZ^n$.
Note that 
$\rho(\ell)=\rho(\ell;\mathbf{0})$.
The quantity $\rho(\ell;\c)$ is a multiplicative function of $\ell$ and our next  result is concerned with estimating it when $\ell=p^r$.

\begin{lemma}\label{lem:rho''}
Let $n=3$
and let  $f\in \ZZ[X_1,\dots,X_n]$ be a  non-singular form of degree $d$.
Then  we have
$
\rho(p^r;\c)\ll_r p^{(1-1/d)r} \gcd(p^r,\c)^3.
$
\end{lemma}

\begin{proof}
As in the proof of Lemma \ref{lem:rho'}, we may proceed under 
the assumption that 
$
p\nmid 2dm\Delta_f \det(Q).
$ 
In particular $f$ is non-singular modulo $p$ and does not vanish identically on the linear form $\c.\x$ modulo $p$.

We begin by proving the result under the assumption that  $p\nmid \c$,  analysing 
$\rho(p^r;\c)$ via 
the non-singular change of variables 
$\y=(x_1,x_2,\c.\x)$. 
Assuming without loss of generality that $p\nmid c_3$, we find that 
$\rho(p^r;\c)$ is equal to
$$
\#\left\{
(y_1,y_2)\bmod{p^r}: \begin{array}{l}
Q(c_3y_1,c_3y_2,-c_1y_1-c_2y_2)\equiv c_3^2m \bmod{p^r}\\
f(c_3y_1,c_3y_2,-c_1y_1-c_2y_2)\equiv 0 \bmod{p^r}
 \end{array}\right\}.
$$
There is no contribution from $y_1,y_2$ for which $p\mid (y_1,y_2)$. Suppose without loss of generality that $p\nmid y_2$. We make the further change of variables $y_1=zy_2$, now finding that the contribution to $\rho(p^r;\c)$ is
$$
\#\left\{
(y_2,z)\bmod{p^r}: p\nmid y_2, ~ y_2^2h(z)\equiv c_3^2m\bmod{p^r}, ~
g(z)\equiv 0\bmod{p^r}
\right\},$$
where
$$
h(z)=Q(c_3z,c_3,-c_1z-c_2) \quad \text{ and } \quad 
g(z)=f(c_3z,c_3,-c_1z-c_2). 
$$
Here $g(z)$ is a polynomial of degree at most $d$ which does not vanish identically modulo $p$. Moreover we are only interested in roots of $g(z)$ modulo $p^r$ for which $p\nmid h(z)$.  It follows from 
work of Stewart \cite[Cor.~2]{stewart}
that the number of $z\bmod{p^r}$ is 
$O(p^{(1-1/d)r})$. For given $z$ there are then at most $2$ available choices for $y_2$, 
which therefore completes the proof of the lemma when $p\nmid \c$.

Suppose now that $p^j\|\c$. If $j\geq r$ then we 
get a satisfactory bound for the lemma by taking  the trivial bound
$\rho(p^r;\c)\leq p^{3r}= \gcd(p^r,\c)^3$.
Alternatively, if $j<r$ we write $\x=\u+p^{r-j}\v$ for $\u \bmod{p^{r-j}}$ and $\v \bmod{p^j}$.
The number of $\u$ is precisely $\rho(p^{r-j};\tilde \c)$, where $\tilde\c=p^{-j}\c$.
The number of $\v$ is trivially at most $p^{3j}$.
Hence we have 
$
\rho(p^r;\c)\leq p^{3j}\rho(p^{r-j};\tilde \c).
$
Applying our earlier bound for the case $p\nmid \c$, we therefore complete the proof of the lemma.
\end{proof}

\subsection{Summary of the argument}\label{s:onion}
It is now time to   survey the key steps in the
proof of Theorems \ref{thm:2large}, \ref{thm:2} and \ref{thm:1}. Let $r\geq 2$ and let 
$f\in \ZZ[X_1,\ldots,X_n]$ be a non-singular form of degree $d\geq 1$. 
Recalling the counting function 
$N_r(Y,f;H)$ from   \eqref{eq:london0}, we let $\Delta>0$ and consider the contributions $N^{(1)}(H)$ and $N^{(2)}(H)$ that were defined in 
\eqref{eq:n1} and \eqref{eq:leek2}, respectively. 
In view of 
Remark \ref{r:quad}, 
we see that the treatment of  
$N^{(1)}(H)$ (i.e.\ the 
small moduli) is handled by Proposition \ref{lem:dynamics} and \eqref{eq:final0}, in which we can take 
$\mu_\infty(Y;H) \ll H^{n-2}$
and $\mu_\infty(\mathcal{O}_\infty(H)) \ll H^{n-2}$,  $\dim(G)=n(n-1)/2$ and $\dim(L)=(n-1)(n-2)/2$.

When $n\geq 4$,   $L$ is simply connected and it   follows
from Proposition \ref{lem:dynamics}
that there exists $\delta>0$  such that 
\begin{align*}
N^{(1)}(H)=~&\mathfrak{S}(Y,f,r) \mu_{\infty}(Y;H)\\
& +
O\left(H^{n-2}\left\{
H^{-\Delta/2}+H^{\Delta(1+ r(\frac{3}{2}n^2-\frac{3}{2}n+1)-\delta}
\right\}
\right).
\end{align*}
Moreover, for $n=3$, there exists $\delta>0$ such that 
\begin{align*}
N^{(1)}(H)=&
\sum_{\mathcal{O}_{\A}\subset Y(\A)}\delta(\mathcal{O}_{\A})\mathfrak{S}(\mathcal{O}_{\A},f,r) \mu_{\infty}(\mathcal{O}_{\A};H)\\
&+
O\left(H^{1-\Delta/2}+H^{1+\Delta(1+ 10r)-\delta}
\right),
\end{align*}
where the sum is taken over finitely many orbits $\mathcal{O}_{\A}$
that have non-trivial intersection with $Y(\RR)\times \prod_{p<\infty} Y(\ZZ_p)$.
Our additional assumption on existence of $r$-free points (when $n=3$)
guarantees that for at least one of the orbits $\mathcal{O}_{\A}$ we have
$\delta(\mathcal{O}_{\A})>0$ and $\mathfrak{S}(\mathcal{O}_{\A},f,r)>0$.
On taking $\Delta>0$ to be sufficiently small in terms of $\delta$ we can ensure that these error terms are all satisfactory from the point of view of Theorems \ref{thm:2large},  \ref{thm:2} and \ref{thm:1}.

For  given $\Delta>0$, it remains to show that 
there exists  $\eta>0$, depending on $\Delta, r, d$ and $n$, such that 
\begin{equation}\label{eq:magic}
N^{(2)}(H)\ll  H^{n-2-\eta}.
\end{equation}
We shall do so provided that $f$ is a non-singular form of degree $d\geq 1$,  with $r$ satisfying the lower bounds from Theorems \ref{thm:2large}, \ref{thm:2} or \ref{thm:1}, which will thereby  suffice to conclude their proof.
For  
$\ell\in \NN$, 
the estimation of $N^{(2)}(H)$ hinges upon good upper bounds for
\begin{equation}
  \label{eq:N(X)}
U_\ell(H)=\card \{\x\in Y(\ZZ): |\x|\leq H, ~ 0\neq f(\x)\equiv 0\bmod{\ell}
\},
\end{equation}
The following result summarises our treatment of $U_{\ell}(H)$ when $d=1$.

\begin{prop}\label{lem:N(X)}
Suppose that $f\in \ZZ[X_1,\dots,X_n]$ is a linear form.
Let $\ve>0$ and let $\ell\in \NN$.
 Then 
$
U_\ell(H)=O_\ve(\ell^{-1} H^{n-2+\ve}).
$
 \end{prop}

This result will be established in \S \ref{s:large1}.
Applying Proposition \ref{lem:N(X)} with $\ell=k^r$ and $\ve=\Delta/2$, we obtain 
$$
N^{(2)}(H)
\ll H^{n-2+\Delta/2}\sum_{k> H^\Delta} 
\frac{|\mu(k) |}{k^2}
\ll H^{n-2-\Delta/2}.
$$
This is satisfactory for \eqref{eq:magic}.

Estimating  $U_{\ell}(H)$ for $d\geq 2$
is  more difficult.
For 
square-free $k\in \NN$ we deal with this by noting that 
$
U_{k^r}(H)\leq U_{k^{j}}(H),
$
for any $j\leq r$, 
whence
\begin{equation}\label{eq:upper-N}
N^{(2)}(H)
\leq \sum_{H^\Delta<k\ll H^{d/r}} |\mu(k)| 
U_{k^j}(H).
\end{equation}
We will establish the following result in \S \ref{s:large2}.

\begin{prop}\label{lem:N(X)''-4}
Let $n\geq 4$ and 
suppose that $f\in \ZZ[X_1,\dots,X_n]$ is a non-singular
form of degree $d\geq 2$.
Let $\ve>0$ and let $k\in \NN$ be square-free such that $k^{jn/(n-1)}\leq H$.
Then 
$
U_{k^j}(H)=O_{\ve,j}(k^{-j/(n-1)} H^{n-2+\ve}).
$
 \end{prop}

This allows us to establish \eqref{eq:magic} when $n\geq 4$.
We 
make the assumption that $j$ and 
$r$ are chosen so that 
\begin{equation}\label{eq:j1}
\frac{jdn}{(n-1)r}\leq1.
\end{equation}
Then it will follow that $k^{jn/(n-1)}\ll H$
in \eqref{eq:upper-N}, since
$k\ll H^{d/r}$. 
Hence  Proposition~\ref{lem:N(X)''-4} yields 
$$
N^{(2)}(H)
\ll_{\ve,j} H^{n-2+\ve}\sum_{k> H^\Delta} 
\frac{|\mu(k) |}{k^{j/(n-1)}},
$$
for any $\ve>0$.
Here the exponent of $k$  exceeds $1$ if and only if 
$j>n-1$.
We choose
$j=n$, with which choice we can conclude that \eqref{eq:magic} holds
for any  $\eta<\Delta/(n-1)$,
provided that $r$ satisfies \eqref{eq:j1} with $j=n$. But this is equivalent to 
 $r\geq dn^2/(n-1)$, which was one of the assumptions in Theorem 
 \ref{thm:2large}.

When $n=3$ we are not able to get such a good bound for $U_{k^j}(H)$.
The following result will also be established in \S \ref{s:large2}.

\begin{prop}\label{lem:N(X)''}
Let $n=3$ and 
suppose that $f\in \ZZ[X_1,X_2,X_3]$ is a non-singular
form of degree $d\geq 2$.
Let $\ve>0$ and let $k\in \NN$ be square-free such that $k^{4j/3}\leq H$.
Then 
$
U_{k^j}(H)=O_{\ve,j}(k^{-j/(3d)} H^{1+\ve}).
$
\end{prop}

Let us see how this is sufficient to  prove \eqref{eq:magic} when $n=3$. We 
make the assumption that 
$j$ and $r$ are chosen so that 
\begin{equation}\label{eq:priory}
\frac{4jd}{3r}\leq1.
\end{equation}
Then, as before,  it will follow that $k^{4j/3}\ll H$
in \eqref{eq:upper-N}, since
$k\ll H^{d/r}$. 
Hence we may apply Proposition \ref{lem:N(X)''} in 
\eqref{eq:upper-N},
giving 
$$
N^{(2)}(H)
\ll_{\ve,j} H^{1+\ve}\sum_{k> H^\Delta} 
\frac{|\mu(k)|}{
k^{j/(3d)}},
$$
for any $\ve>0$ and any 
$j$ such that \eqref{eq:priory} holds.
Here the exponent of $k$  exceeds $1$ if and only if 
$j>3d$.
We choose $j=3d+1$, with which choice we conclude that \eqref{eq:magic} holds
with $n=3$ and any $\eta<\Delta/(3d)$,
provided that $r$ satisfies the inequality
$
r\geq \tfrac{4}{3}d(3d+1),
$
which was one of the hypotheses of 
 Theorem \ref{thm:2}.

\section{Large moduli}\label{s:largesse}

\subsection{Linear polynomials}\label{s:large1}

In this section we establish 
Proposition \ref{lem:N(X)}.
Let $Y\subset \AA^n$ denote  the quadric 
$Q=m$, for  $n\geq 3$, and let $f\in \ZZ[X_1,\dots,X_n]$ be a linear form.
We are interested in the quantity 
$$
U_\ell(H)=\card \{\x\in Y(\ZZ): |\x|\leq H, ~ 0\neq f(\x)\equiv 0\bmod{\ell}
\},
$$
for  $\ell\in \NN$.  After a non-singular linear change of variables, we see that it suffices to prove Proposition \ref{lem:N(X)}
when  $f= X_n$.  In particular $U_\ell(H)=0$ unless $\ell\ll H$, which we henceforth assume.

Suppose first that $n\geq 4$. 
Note that 
$$
U_\ell(H)= \sum_{\substack{h\in \ZZ, ~ 0<|h|\ll H\\ h\equiv
    0\bmod{\ell}}}L_{h}(H),
$$
where 
$
L_h(H)= 
\card \{\x\in Y(\ZZ): |\x|\leq H, ~ x_n=h
\},
$
for any $h\in \ZZ$.  For $n\geq 4$ we   claim that
\begin{equation}\label{lem:>3}
L_h(H)=O_{\ve}(H^{n-3+\ve}),
\end{equation}
for any $\ve>0$, where the implied constant does not depend on $h$.
This will suffice to establish Proposition \ref{lem:N(X)} for $n\geq 4$, since there are $O(\ell^{-1}H)$ integer values of  $h\ll H$ such that $\ell\mid h$.

To verify \eqref{lem:>3}  we write  
$$
q(X_1,\ldots,X_{n-1})=Q(X_1,\ldots,X_{n-1},h)-m,
$$ 
for the quadratic polynomial  in $L_h(H)$. 
Then 
$$
q_0(X_1,\ldots,X_{n-1})=Q(X_1,\ldots,X_{n-1},0)
$$ 
and it follows 
that $\rank(q_0)\geq n-2$, since $Q$ is non-singular. 
In particular $\rank(q_0)\geq 2$ if $n\geq 4$.
Moreover, if $q$ were reducible, then the quadratic form 
$$
R_h(X_0,\ldots,X_{n-1})= Q(X_1,\ldots,X_{n-1},hX_0)-mX_0^2
$$
would have rank at most $2$.  But this is impossible when $n\geq 4$. Indeed, if 
$h=0$, then  $\rank(R_0)\geq n-1\geq 3$. Equally, if  $h\neq 0$, then  $\rank(R_h)\geq n-1\geq 3$,
since $\rank(A+B)\geq |\rank(A)-\rank(B)|$ for any $n\times n$ matrices with integer coefficients. 
The estimate \eqref{lem:>3} is now a trivial consequence of Theorem \ref{thm:3}.

It remains to establish Proposition \ref{lem:N(X)} when $n=3$ and $f=X_3$.
We may write 
\begin{equation}\label{eq:pisa}
Q(X_1,X_2,X_3)=P(X_1,X_2)+X_3(cX_1+dX_2)+eX_3^2,
\end{equation}
where $P(X_1,X_2)=Q(X_1,X_2,0)$ and $c,d,e\in \ZZ.$
Since $Q$ is non-singular, it follows that $P$ has rank $1$ or $2$. We will need to deal with each of these cases separately. 

Suppose first that $\rank(P)=2$.
Then 
after a non-singular change of variables in $X_1$ and $X_2$ alone it suffices to proceed under the assumption that $c=d=0$ in \eqref{eq:pisa}.
Hence
$$
U_\ell(H)\leq 
\card\left\{
\x\in \ZZ^3: 
\begin{array}{l}
|\x|\ll H, ~ 
0\neq x_3 \equiv 0 \bmod{\ell}\\
P(x_1,x_2)
=m-e x_3^2
\end{array}
\right\}.
$$
Let $h\in \ZZ$ such that $h\ll H$ and  $0\neq h\equiv 0\bmod{\ell}$.
There are $O(\ell^{-1}H)$ such integers. 
If $m-e h^2\neq 0$ then an application of Lemma \ref{lem:affquad} 
shows that there are $O_\ve(H^\ve)$ choices for $x_1,x_2\ll H$ such that 
$P(x_1,x_2)=m-e h^2$.
If $m-e h^2=0$, which can happen for at most $2$ values of $h$, 
we deduce that $0\neq m\equiv 0\bmod{\ell}$.
Hence $\ell=O(1)$  and there are $O(H)$ choices for  $x_1,x_2\ll H$ such that 
$P( x_1, x_2)=0$. This case therefore contributes $O(\ell^{-1}H)$ overall, which shows that 
Proposition \ref{lem:N(X)} holds when $n=3$ 
and $\rank(P)=2$.

Next we suppose that $\rank(P)=1$ in \eqref{eq:pisa}, still with $n=3$ and $f=X_3$.  Then 
$Q$ takes the shape
$$
Q(X_1,X_2,X_3)=aL_1(X_1,X_2)^2+X_3L_2(X_1,X_2)+eX_3^2,
$$
for $a,e\in \ZZ$ and linear forms $L_1,L_2\in \ZZ[X_1,X_2]$. Moreover  $aL_1$ and $L_2$ are non-zero and non-proportional,  since $Q$ is non-singular. 
After a change of variables we obtain 
$$
U_\ell(H)\leq 
\card\left\{
\x\in \ZZ^3: 
\begin{array}{l}
|\x|\ll H, ~ 
0\neq x_3 \equiv 0 \bmod{\ell}\\
x_1^2+x_2x_3+cx_3^2=d
\end{array}
\right\},
$$
for suitable integers $c,d=O(1)$ with $d\neq 0$.  
If $x_3=0$ then $\ell=O(1)$
and there are clearly $O(H)$ choices for $x_1,x_2$. Hence there is a contribution of 
$O(\ell^{-1}H)$ to $U_\ell(H)$ from this case. Alternatively, the contribution from non-zero $x_3$
is at most
$$
\sum_{
\substack{0<|x_3|\ll H\\
x_3\equiv 0 \bmod{\ell}}}
\card\{x_1\ll H:
x_1^2\equiv d\bmod{x_3}
\}
\ll H \sum_{
\substack{0<|x_3|\ll H\\
x_3\equiv 0 \bmod{\ell}}}
\frac{\nu(x_3;d)}{|x_3|},
$$
where 
$$
\nu(q;d)=\#\{n\bmod{q}: n^2\equiv d \bmod{q}\}.
$$
On noting that $\nu(q;d)$ is a multiplicative function of $q$, with 
$\nu(2^j;d)\leq 4$ and 
$\nu(q;d) =\sum_{k\mid q} |\mu(k)| (\frac{d}{q})$ when $q$ is odd, it  follows that
$
\nu(q;d)\leq 2^{\omega(q)+1}\ll_\ve q^{\ve/2},
$
where $\omega(q)$ denotes the number of distinct prime factors of $q$.
Next we observe that 
$$
\sum_{
\substack{0<|x_3|\ll H\\
x_3\equiv 0 \bmod{\ell}}}
\frac{1}{|x_3|}\ll \ell^{-1}\log H.
$$
Thus the non-zero $x_3$ contribute 
$O_\ve(\ell^{-1}H^{1+\ve})$ overall. The same bound therefore holds for $U_\ell(H)$ when $n=3$
and $\rank(P)=1$, which thereby concludes the proof of Proposition \ref{lem:N(X)}.

\subsection{Higher degree polynomials}\label{s:large2}

We now place ourselves in the setting of Propositions \ref{lem:N(X)''-4} 
and \ref{lem:N(X)''}. 
Let   $f\in \ZZ[X_1,\ldots,X_n]$ be a  non-singular form of degree $d\geq 2$.
If $\Delta_f$ is the discriminant of $f$, then $\Delta_f$ and $m\det(Q)$ are both non-zero integers.  
Let $\ve>0$ be given once and for all.
For square-free  $k\in \NN$ such that $k^{2j}\leq H$, 
we  want to  estimate the quantity $U_{k^j}(H)$ in 
\eqref{eq:N(X)}.

Let us put $\ell=k^j$ for convenience. 
Our estimation of $U_{\ell}(H)$
is inspired by an argument of Browning and   Munshi \cite[Lemma~4]{BM}.
We will require some elementary facts about integer sublattices, as established by Davenport 
 \cite[Lemma 5]{dav}.
Suppose that  $\sfl\subset \ZZ^n$ is a lattice of rank $r$ and determinant 
$\det(\sfl)$.
Then there exists a ``minimal'' basis
$\ma{m}^{(1)},\ldots,\ma{m}^{(r)}$  of $\sfl$ 
with the property
$$
\lambda_{j}\ll  |\u|/|\ma{m}^{(j)}|, \quad (1\leq j\leq r),
$$
whenever 
$\u\in\sfl$ is written as 
$
\u=\sum_{j=1}^r\lambda_{j}\ma{m}^{(j)}.
$
Furthermore,  the basis is constructed  in such a way that
$1\leq |\ma{m}^{(1)}|\leq \cdots \leq |\ma{m}^{(r)}|$ and 
\begin{equation}\label{eq:train}
\det(\sfl)\leq \prod_{j=1}^{r}|\ma{m}^{(j)}|\ll \det(\sfl).
\end{equation}

Breaking into residue classes modulo $\ell$, we obtain
\begin{equation}\label{m:6}
U_\ell(H)\leq 
\hspace{-0.3cm}
\sum_{\substack{\bxi\bmod{\ell}\\  Q(\bxi)\equiv m \bmod{\ell} \\
 f(\bxi)\equiv 0 \bmod{\ell}
}}
\hspace{-0.3cm}
\#\{
 \x\in Y(\ZZ):  |\x|\leq H, ~\x\equiv \bxi \bmod{\ell}
\}.
\end{equation}
We  denote the set whose cardinality appears in the inner sum by $S_\ell(H;\bxi)$.
If $S_\ell(H;\bxi)$ is empty then there is nothing to prove. Alternatively, suppose we are given $\x_0\in S_\ell(H;\bxi)$. Then any other vector in the set must be congruent to $\x_0$ modulo $\ell$. 
Making the change of variables 
$\x=\x_0+\ell \y$ in $S_\ell(H;\bxi)$, we have  $|\y|<2\ell^{-1}H$.  
Furthermore, 
\begin{equation}\label{m:reich}
\y.\nabla Q(\x_0) +\ell Q(\y)=0,
\end{equation}
by Taylor's formula, 
since $Q(\x_0+\ell \y)=m$ and $Q(\x_0)=m$.
Note here that $\nabla Q(\x_0)\neq \mathbf{0}$ for any $\x_0\in Y(\ZZ)$. 

When $n\geq 4$, it will be convenient to 
deal separately with the contribution from $\y$ for which $\y.\nabla Q(\x_0)=0$. 
Using this linear equation to eliminate one of the variables, we arrive at a quadratic form in $n-1$ variables. We claim that this quadratic form is non-singular. When $n\geq 4$ this automatically implies that it is also absolutely irreducible.
To see the claim, suppose that $\mathbf{B}$ is the underlying symmetric  matrix associated to $Q$, so that 
$\nabla Q(\y)=2\mathbf{By}$ and 
$\nabla Q(\x_0)=2\mathbf{B}\x_0$. Then if the quadric obtained from $\mathbf{Y}.\nabla Q(\x_0)=Q(\mathbf{Y})=0$ is singular, there must exist $\lambda,\mu\in \overline\QQ$ and $\y\neq \mathbf{0}$, such that  $(\lambda,\mu)\neq (0,0)$ and 
$$
Q(\y)=0, \quad 2\mathbf{B}(\lambda \y+\mu \x_0)=\mathbf{0}.
$$ 
Since $\mathbf{B}$ is non-singular, this implies that 
$\lambda \y+\mu \x_0=\mathbf{0}$, which is impossible since $0\neq m=Q(\x_0)$.
Hence the claim follows and Theorem~\ref{thm:3} implies that for $n\geq 4$
the overall contribution to $S_\ell(H;\bxi)$  from $\y$ for which $\y.\nabla Q(\x_0)=0$ is 
$\ll_\ve (H/\ell)^{n-3+\ve}$. 

Returning to general $n\geq 3$,  \eqref{m:reich} implies that the $\y$ under consideration satisfy 
the congruence 
$
\y.\nabla Q(\bxi)\equiv 0\bmod{\ell},
$
since 
$\x_0\equiv \bxi \bmod{\ell}$.
Let us write
 $$\sfl_{\bxi}=\{\y\in \ZZ^n:  \y.\nabla Q(\bxi)\equiv 0 \bmod{\ell}\}.
 $$
Our work so far has shown that 
$$
\#S_\ell(H;\bxi)\leq 2\#\{
 \y\in\sfl_{\bxi} :  |\y|<2\ell^{-1}H , ~\text{\eqref{m:reich} holds}\},
$$
when $n=3$, and 
$$
\#S_\ell(H;\bxi)\ll_\ve \left(\frac{H}{\ell}\right)^{n-3+\ve} +\#\left\{
 \y\in\sfl_{\bxi} :  
 \begin{array}{l}
 |\y|<2\ell^{-1}H , ~\text{\eqref{m:reich} holds} \\
 \y.\nabla Q(\x_0)\neq 0
 \end{array}
 \right\},
$$
when $n\geq4$.

The set $\sfl_{\bxi}$ defines an integer lattice of rank $n$.
To calculate its determinant we
write $\hat \ell=\ell/\gcd(\ell,\nabla Q(\bxi))$ and note that $\hat\ell\ZZ^n\subset \sfl_{\bxi}$.
Thus we have 
$$
\det(\sfl_{\bxi})=[\ZZ^n:\sfl_{\bxi}]=\frac{[\ZZ^n: \hat \ell\ZZ^n]}{[\sfl_{\bxi}: \hat \ell \ZZ^n]}.
$$
The numerator here is clearly $\hat \ell^n$ and the denominator is seen to be $\hat\ell^{n-1}$, since there are 
$\hat\ell^{n-1}$ distinct values of $\y \bmod{\hat\ell}$ for which $\y.\nabla Q(\bxi)\equiv 0\bmod{\ell}$.
Thus $\det(\sfl_{\bxi})=\hat\ell$.
We claim that in fact
$
\ell \ll \det( \sfl_{\bxi}) \leq \ell,
$
the upper bound being  trivial. For the lower  bound, note that $\bxi.\nabla Q(\bxi)= 2Q(\bxi)\equiv 2m\bmod{\ell}$
in \eqref{m:6}, 
whence
$ \gcd(\ell,\nabla Q(\bxi))\ll 1$.

Let $\M$ denote the non-singular matrix formed from taking column vectors to be a minimal basis  $\m_1,\ldots,\m_n$ for $\sfl_{\bxi}$. 
Making the change of variables $\y=\M\bla$, 
we arrive at the  equation $q(\bla)=0$, where 
if
$b_i=\ell^{-1}\m_i.\nabla Q(\x_0)$ for $i=1,\dots,n$, then 
\begin{equation}\label{eq:q}
q(\bla)=Q(\lambda_1 \m_1+\dots+\lambda_n \m_n)+b_1\lambda_1+\dots+
b_n\lambda_n.
\end{equation}
This is 
obtained from 
\eqref{m:reich} by substitution and dividing through by $\ell$. It is defined over $\ZZ$ and 
the quadratic homogeneous part $q_0$ has  underlying matrix $\M^T \B \M$
of full rank  $n\geq 3$,
where  $\B$ is the matrix
associated to  $Q$. 
Our argument now diverges according to whether $n\geq 4$ or $n=3.$

\subsubsection*{The case $n\geq 4$}

Recalling the definition of $S_\ell(H;\bxi)$, our main aim in this section is to establish the
following result, which may prove to be of independent interest. 

\begin{prop}\label{prop:7.1}
Assume that $n\geq 4$ and let $\ve>0$. Then we have 
$$
\#\{\x\in  Y(\ZZ):  |\x|\leq H, ~\x\equiv \bxi \bmod{\ell}\}\ll_{\ve} \left(\frac{H}{\ell}\right)^{n-3+\ve}\left(1+\frac{H}{\ell^{n/(n-1)}}\right),
$$
for any $\ell\in \NN$ and any $\bxi\in (\ZZ/\ell\ZZ)^n$, where the implied constant is uniform in $\ell$ and $\bxi$.
\end{prop}

Before proving this result let us see how it suffices to 
complete the proof of Proposition \ref{lem:N(X)''-4}.
Taking $\ell=k^j$ for square-free $k\in \NN$,
we have  $\ell^{n/(n-1)}\leq H$, by hypothesis.  Thus the second term  dominates the first term in Proposition~\ref{prop:7.1}. 
Substituting this into \eqref{m:6} and applying 
Lemma \ref{lem:rho'},
we finally arrive at the bound
$$
U_\ell(H) \ll_\ve  
\frac{\rho(\ell)H^{n-2+\ve}}{\ell^{n-2+1/(n-1)+\ve}}
\leq
\frac{C_j^{\omega(\ell)}H^{n-2+\ve}}{\ell^{1/(n-1)+\ve}},
$$
for an appropriate constant $C_j>0$ depending on $j$. Taking 
$C_j^{\omega(\ell)}=O_{\ve,j}(\ell^{\ve})$, 
we therefore conclude the proof of Proposition \ref{lem:N(X)''-4}.

\begin{proof}[Proof of Proposition \ref{prop:7.1}]
We employ the  properties of the minimal basis that were recorded above. This leads to  the inequality
$$
\#S_\ell(H;\bxi)\ll_\ve \left(\frac{H}{\ell}\right)^{n-3+\ve} +\#\left\{
 \bla\in\ZZ^n: 
 \begin{array}{l}
\text{$\lambda_i\ll 
(|\m_i| \ell)^{-1}H$ for $1\leq i\leq n$} \\q(\bla)=0,~
b_1\lambda_1+\dots+ b_n\lambda_n\neq 0
 \end{array}
 \right\}.
$$
We begin by considering $q(\bla)$ when  $\lambda_n=h$ is a fixed integer, with  $h\neq 0$ when $b_1=\dots=b_{n-1}=0$.
Put 
$$
r(X_1,\dots,X_{n-1})=q(X_1,\dots,X_{n-1},h).
$$ 
 We claim that $r$ is absolutely irreducible and that $\rank(r_0)\geq 2$.
The latter follows on noting that 
$r_0(X_1,\dots,X_{n-1})=Q(X_1\m_1+\dots+X_{n-1}\m_{n-1})$, which must have rank at least $n-2\geq 2$. To check  that $r$ is absolutely irreducible we consider
 the rank of the  quadratic form
$
X_0^2r(X_1/X_0,\dots,X_{n-1}/X_0)$. The latter  is equal to
$$
Q(X_1 \m_1+\dots+X_{n-1} \m_{n-1}+hX_0\m_n)+X_0\left(b_1X_1+\dots+
b_{n-1}X_{n-1}+b_nhX_0\right).
$$
If $h=0$ then this has rank at least  $ n-2+1=n-1\geq 3$ since  in this scenario $(b_1,\dots,b_{n-1})\neq (0,\dots,0)$. If $h\neq 0$ then it clearly has rank at least $n-1\geq 3$.
Hence $r$ is indeed absolutely  irreducible.

Returning to our estimation of
$\#S_\ell(H;\bxi)$, suppose first that 
$|\m_n| \ell\gg H$, so that $\lambda_n=0$ in any solution to be counted. 
It follows from the  condition $b_1\lambda_1+\dots+ b_n\lambda_n\neq 0$ that we may assume 
$(b_1,\dots,b_{n-1})\neq (0,\dots,0)$. 
Then $q(X_1,\dots,X_{n-1},0)$ satisfies the hypotheses of Theorem \ref{thm:3} and we see that there is an overall contribution of $O_\ve((H/\ell)^{n-3+\ve})$ from this case, which is satisfactory. 

We proceed under the assumption that 
$|\m_i| \ell\ll H$ for $1\leq i\leq n$.
We will  fix a value of $\lambda_n$ and then use Theorem \ref{thm:3} to estimate the associated number of $\lambda_1,\dots,\lambda_{n-1}$.
It follows from the  condition $b_1\lambda_1+\dots+ b_n\lambda_n\neq 0$ that when $b_1=\dots=b_{n-1}=0$, any solution with $\lambda_n=0$ is to be ignored. 
Let $\lambda_n=h$ be fixed and put $r(X_1,\dots,X_{n-1})=q(X_1,\dots,X_{n-1},h)$, as before.  Then $r$ satisfies the hypotheses of  Theorem \ref{thm:3}  and we  deduce that the total number of 
$\lambda_1,\dots,\lambda_{n-1}$ associated to $h$ is
$$
\ll_\ve \left(\frac{H}{|\m_1|\ell}\right)^{n-3+\ve}.
$$
The implied constant in this estimate depends at most on $\ve$ and $n$ and, crucially,  is independent of $h$.
Summing over $h$ therefore leads to the overall conclusion that 
$$
\#S_\ell(H;\bxi)\ll_\ve \left(\frac{H }{\ell}\right)^{n-3+\ve}+
\frac{H^{n-2+\ve}}{\ell^{n-2+\ve} |\m_1|^{n-3}|\m_n|}.
$$
It follows from  \eqref{eq:train} that 
\begin{align*}
|\m_1|^{n-3}|\m_n|\geq 
|\m_1|^{1/(n-1)}|\m_n|\geq 
 (|\m_1|\dots|\m_n|)^{1/(n-1)}
 &\gg (\det \sfl_{\bxi})^{1/(n-1)}\\
&\gg \ell^{1/(n-1)}.
\end{align*}
Taking this lower bound in our estimate for 
$\#S_\ell(H;\bxi)$ concludes the proof.
\end{proof}

\subsubsection*{The case $n=3$}

We henceforth take $n=3$ and concern ourselves with the  proof of Proposition  \ref{lem:N(X)''}.
For $j\geq2$ 
we are interested in estimating the quantity
$U_{k^j}(H)$,
when $k^{4j/3}\leq H$.
Developing 
$U_{k^j}(H)$ as in \eqref{m:6}, our 
starting point is the inequality
$$
\#S_\ell(H;\bxi)\leq 2\#\{
 \bla\in\ZZ^3 :  \text{$\lambda_i\ll 
(|\m_i| \ell)^{-1}H$ for $1\leq i\leq 3$} ,~ q(\bla)=0\},
$$
where $q$ is given by $\eqref{eq:q}$ and $\ell=k^j$.
When $n=3$ we have been unable to 
produce a satisfactory estimate for this quantity by first fixing $\lambda_3$ and analysing the resulting binary quadratic polynomial.  The problem is that for certain choices of $\lambda_3$ it may happen that the resulting polynomial is reducible over $\QQ$, which would then contribute too much to $\#S_\ell(H;\bxi)$.   Instead we will apply Theorem \ref{thm:3} directly to the ternary quadratic polynomial and attempt to show that $|\m_1|$ cannot often be very small.

Since $q_0$ has rank $3$ it follows that 
$q$ is absolutely irreducible. 
By \eqref{eq:train} one has $|\m_1|^3\ll \det (\sfl_{\bxi})\leq \ell$. Moreover, 
$\ell^{4/3}=k^{4j/3}\leq H$.
Hence $H\gg |\m_1|\ell$ and so
  Theorem~\ref{thm:3}
implies that 
\begin{equation}\label{eq:now}
\#S_\ell(H;\bxi)\ll_\ve \left(\frac{H}{|\m_1| \ell}\right)^{1+\ve}.
\end{equation}
This estimate is sharpest when 
$|\m_1|^3$ has exact order $\ell$. 
Let us put
$$
\kappa=\frac{1}{3d},
$$
for convenience.
Returning to \eqref{m:6} we will denote by $U_\ell^{(I)}(H)$ the overall contribution  to the right hand side from $\bxi$ for which the corresponding basis vector $\m_1=\m_1(\bxi)$ satisfies 
$|\m_1|\geq \ell^\kappa$.  
We denote by $U_\ell^{(II)}(H)$ the  contribution  from  $\bxi$ for which the smallest basis vector satisfies
$|\m_1|< \ell^\kappa$.  
The estimation of $U_\ell^{(I)}(H)$  is a straightforward consequence of 
\eqref{m:6}, \eqref{eq:now} and Lemma \ref{lem:rho'}. 
Thus  we deduce that 
$$U_\ell^{(I)}(H)\ll_\ve \frac{\rho(\ell)H^{1+{\ve}}}{\ell^{(1+\kappa)(1+\ve)}}
\leq \frac{C_j^{\omega(\ell)}H^{1+{\ve}}}{\ell^{\kappa+\ve}}
\ll_{\ve,j} \frac{H^{1+{\ve}}}{\ell^{\kappa}},
$$
for an appropriate constant $C_j>0$ depending on $j$. 
In view of our choice of $\kappa$, 
this is clearly satisfactory for 
Proposition  \ref{lem:N(X)''}.

In order to estimate 
$U_\ell^{(II)}(H)$, 
we introduce 
averaging 
over the least non-zero elements of $\sfl_{\bxi}$.
Thus, given a non-zero vector $\m\in \ZZ^3$ satisfying  $|\m|< \ell^{\kappa}$, 
we shall return to \eqref{m:6} and consider the overall contribution from 
$\bxi$ and $\x$ such that  $\m.\nabla Q(\bxi)\equiv 0\bmod{\ell}$.
Note that $\m.\nabla Q(\bxi)=(2\B\m).\bxi$, where $\B$ is the symmetric matrix associated to $Q$.
Applying \eqref{eq:now}, we obtain 
$$
U_\ell^{(II)}
(H)\ll_\ve
\sum_{
\substack{
\m\in \ZZ^3\\ 0<|\m|<\ell^{\kappa}
}} 
\rho(\ell;2\B\m)
\left(\frac{H}{|\m| \ell}\right)^{1+\ve}.
$$
Invoking  Lemma \ref{lem:rho''}, we deduce the existence of a constant $C_j>0$ such that 
\begin{align*}
\rho(\ell;2\B\m)
  \leq C_j^{\omega(\ell)} \ell^{1-1/d}
\gcd( \ell,2\B\m)^3.
\end{align*}
Thus 
\begin{align*}
U_\ell^{(II)}
(H)
 & \ll_\ve \frac{C_j^{\omega(\ell)} H^{1+\ve}}{\ell^{1/d+\ve}}
\sum_{
\substack{
\m\in \ZZ^3\\ 0<|\m|<\ell^{\kappa}
}} 
\frac{\gcd( \ell,2\B\m)^3}{|\m|}.
\end{align*}
The contribution to the inner sum from $\m$ such that $M<|\m|\leq 2M$ is 
at most
\begin{align*}
\sum_{\substack{h\mid \ell}}
\frac{h^3}{M}\#\{
\m\in \ZZ^3:
M<|\m|\leq 2M, ~
h\mid 2\B\m\}
&\ll 
\sum_{\substack{h\mid \ell}}
\frac{h^3}{M}
\left(\frac{M}{h}\right)^3\\
&\ll \tau(\ell)M^2.
\end{align*}
Summing over  dyadic intervals for $M<\ell^\kappa$, we now find that 
\begin{align*}
U_\ell^{(II)}(H)
 & \ll_\ve \frac{\tau(\ell)C_j^{\omega(\ell)} H^{1+\ve}}{\ell^{1/d+\ve}}
\sum_{M<\ell^\kappa} 
M^2\ll_{\ve,j}  \ell^{2\kappa-1/d}H^{1+\ve}.
\end{align*}
Recalling 
 our choice of $\kappa$, 
this is also  satisfactory for 
Proposition  \ref{lem:N(X)''} and so completes its proof.


\begin{thebibliography}{99}

\bibitem{baker1}
R.C.\ Baker, The values of a quadratic form at square-free points. {\em Acta Arith.} {\bf 124} (2006), 
101--137.

\bibitem{bo} Y.\ Benoist and  H.\ Oh, Effective equidistribution of $S$-integral points on symmetric varieties. {\em Ann. Inst. Fourier} {\bf 62} (2012), 1889--1942.

\bibitem{berg} M. Berger, Les espaces sym\'etriques noncompacts.  {\em Ann. Sci. \'Ecole Norm. Sup.} {\bf 74} (1957), 85--177.

\bibitem{bh} A.\ Borel and  Harish-Chandra, Arithmetic subgroups of algebraic groups.
{\em Ann. Math.} {\bf 75} (1962),  485--535.


\bibitem{borovoi'}
M.\ Borovoi,
On representations of integers by indefinite ternary quadratic forms. 
{\em J.\ Number Theory }  {\bf 90} (2001)  281--293.



\bibitem{borovoi}
M.\ Borovoi and Z.\ Rudnick,
Hardy--Littlewood varieties and semisimple groups.
{\em Invent.\ Math.\ }  {\bf 119} (1995)  37--66.

\bibitem{BM}
T.D.\ Browning
and  R.\ Munshi,
Rational points on singular intersections of quadrics.
{\em Compositio Math.} 
{\bf 149} (2013), 1457--1494.



\bibitem{pila}
T.D.\ Browning, D.R.\ Heath-Brown and P.\ Salberger, 
Counting rational points on algebraic varieties.
{\em Duke Math.\ J.} {\bf 132} (2006), 545--578.

\bibitem{bs} M.\ Burger and P.\ Sarnak, Ramanujan duals. II. {\em Invent.\ Math.} {\bf 106} (1991), 1--11.	

\bibitem{c} L.\ Clozel, D\'emonstration de la conjecture $\tau$. {\em Invent.\ Math.} {\bf 151} (2003), 297--328.

\bibitem{cx} J.-L. Colliot-Th\'el\`ene and F. Xu, Brauer--Manin obstruction for integral
points of homogeneous spaces and representation by integral quadratic forms,
{\em Compositio Math.} {\bf 145} (2009), 309--363.

\bibitem{dav}
H.\ Davenport, Cubic forms in 16 variables. {\em Proc. Roy. Soc.
A} {\bf 272} (1963), 285--303.

\bibitem{drs} W.\ Duke, Z.\ Rudnick and P.\ Sarnak, Density of integer points on affine homogeneous varieties. {\em Duke Math. J.} {\bf 71} (1993), 143--179.


\bibitem{sfree}
P.\ Erd\H{o}s, Arithmetical properties of polynomials.
{\em J. London Math. Soc.} {\bf 28} (1953), 416--425.

\bibitem{em} A.\ Eskin and  C.\ McMullen,
Mixing, counting, and equidistribution in Lie groups.
{\em Duke Math. J.} {\bf 71} (1993), 181--209.

\bibitem{ems} A.\ Eskin, S.\ Mozes, and N.\ Shah, Unipotent flows and counting lattice points on homogeneous varieties. {\em Ann.\ Math.} {\bf 143} (1996), 149--159.


\bibitem{F-L}
W.\ Fulton and R.\ Lazarsfeld, 
Connectivity and its applications in algebraic geometry. {\em Algebraic geometry (Chicago, Ill., 1980)}, 26--92,
Lecture Notes in Math. {\bf 862}, Springer, 1981. 




\bibitem{annals} D.R. Heath-Brown,
The density of rational points on curves and surfaces. {\em
Ann.\  Math.} {\bf 155} (2002), 553--595.


\bibitem{gos} A.\ Gorodnik, H.\ Oh and N.\ Shah, Integral points on symmetric varieties and Satake compactifications. {\em Amer. J. Math.} {\bf 131} (2009), 1--57.

\bibitem{km} D.\ Kleinbock and G.\ Margulis, Logarithm laws for flows on homogeneous spaces. {\em Invent.\ Math.} {\bf 138} (1999), 451--494.


\bibitem{n-s}
A.\ Nevo and  P.\ Sarnak, 
Prime and almost prime integral points on principal homogeneous spaces.
{\em Acta Math.}
{\bf 205} (2010), 361--402.

\bibitem{pila-ast} J.\ Pila, Density of integral and rational points on
varieties. {\em Ast\'{e}risque} {\bf 228} (1995), 183--187.

\bibitem{poonen}
B.\ Poonen,
Squarefree values of multivariable polynomials.
{\em Duke Math. J.}
{\bf 118} (2003), 189--373.

\bibitem{sch} H.\ Schlichtkrull, {\em Hyperfunctions and harmonic analysis on symmetric spaces}. Progress in Mathematics  {\bf 49}, Birkha\"user Boston,  Boston, MA, 1984.

\bibitem{stein}  R. Steinberg, {\em Endomorphisms of linear algebraic groups}, Memoirs of the American Mathematical Society {\bf 80}, American Mathematical Society, Providence, R.I., 1968.

\bibitem{stewart}  C.L.\ Stewart, On the number of solutions of polynomial
  congruences 
and Thue equations. {\em J. Amer. Math. Soc.}  {\bf 4}  (1991),
793--835.  


\bibitem{swd}
H.P.F.\ Swinnerton-Dyer,  Rational zeros of two quadratic forms.
{\em Acta Arith.} {\bf 9} (1964), 261--270. 


\end{thebibliography}
\end{document}